\documentclass[11pt, a4paper]{article}

\usepackage[cp1250]{inputenc}
\usepackage{lmodern}
\usepackage{a4wide}
\usepackage{amsfonts}
\usepackage{bbm}
\usepackage[leqno]{amsmath}
\usepackage{amssymb}
\usepackage{amsthm}
\usepackage{mathrsfs}
\usepackage{ifpdf}
\usepackage[dvips]{graphicx}
\usepackage{epsfig}
\usepackage{latexsym}
\usepackage{authblk}
\usepackage[mathscr]{euscript}
\usepackage[T1]{fontenc}
\usepackage{graphics}
\usepackage{enumerate}
\usepackage{longtable}
\usepackage{fancyhdr}
\usepackage{booktabs} 
\usepackage{multicol} 
\usepackage{multirow}
\usepackage[bookmarksopen]{hyperref}
\usepackage{natbib}
\newcommand{\N}{\mathbb{N}}

\newcommand{\E}{\mathrm{E}}

\newcommand{\Z}{\mathbb{Z}}

\newcommand{\p}{\mathbb{P}}
\newcommand{\q}{\mathbb{Q}}

\newcommand{\dd}{\mathrm{d}}
\newcommand{\im}{\mathrm{i}}
\newcommand{\Pcal}{\mathcal{P}}
\newcommand{\Qcal}{\mathcal{Q}}
\newcommand{\Rcal}{\mathcal{R}}
\newcommand{\sds}{\mathrm{SDS}}

\newcommand{\ds}{\mathrm{DS}}
\newcommand{\pds}{\mathrm{PDS}}

\newcommand{\s}{\mathrm{S}}

\newcounter{citac}[section] % theorem

\numberwithin{equation}{section}

\theoremstyle{plain}
\newtheorem{theorem}{Theorem}[section]
\newtheorem{proposition}[theorem]{Proposition}
\newtheorem{corollary}[theorem]{Corollary}
\newtheorem{lemma}[theorem]{Lemma}

\theoremstyle{definition}
\newtheorem{definition}[theorem]{Definition}
\newtheorem{property}[theorem]{Property}
\newtheorem{remark}[theorem]{Remark}

\begin{document}

\title{Generalized definitions of discrete stability}
\renewcommand\Authands{ and }

\author[a]{Lenka Sl\'amov\'a \thanks{Corresponding author; Email: \texttt{slamova.lenka@gmail.com};}}
\author[a]{Lev B. Klebanov \thanks{Email: \texttt{levbkl@gmail.com}}}

\affil[a]{\small{\textit{Department of Probability and Mathematical Statistics, Charles University in Prague, Czech Republic}}}

\maketitle

\abstract{This article deals with different generalizations of the discrete stability property. Three possible definitions of discrete stability are introduced, followed by a study of some particular cases of discrete stable distributions and their properties.}

\section{Introduction}

Stability in probability theory refers to a~property of probability distributions when a~sum of normalized, independent and identically distributed (i.i.d.) random variables has the same distribution (up to scale and shift) no matter how many summands we consider. Random variables with this property are called stable and they form a~wide class of probability distributions. Except for one particular case, the Gaussian distribution, all stable distributions are heavy tailed. The classical stability refers to stability under summation but the concept can be extended onto other systems as well. Stability under maxima (or max-stability) leads to heavy tailed distributions called generalized extreme value distributions; stability under random summation where the number of summands is a~random variable leads to heavy tailed $\nu$-stable distributions. Stability of discrete systems is a~topic that has not been studied as extensively as others but here also the discrete stable distributions exhibit heavy tails. 

Introduced by Paul L\'evy in \citep{levy}, stable distributions are a~generalization of Gaussian distribution in several ways. The theory of stable distributions was developed in monographs by \cite{levy1937} and \cite{khintchine}, and further extended in the work by \cite{gnedenko_kolmogorov_russian} and \cite{feller2}. There exist few equivalent definitions of stable distributions. Paul L\'evy defined stable distributions by specifying their characteristic function. For that he used the L\'evy-Khintchine representation of infinitely divisible distributions. Second definition is connected to the ``stability'' property -- a~sum of stable random variables is again a~stable random variable, a~well known property of Gaussian random variables. Third is the generalized central limit theorem -- stable distributions appear as a~limit of sums of independent and identically distributed random variables without the standard assumption of the central limit theorem about finite variance.  This result generalizes the central limit theorem and is due to \cite{gnedenko_kolmogorov_russian}. Gaussian distribution is a~special (limit) case of stable distributions, the only stable law with finite variance. Recent and extensive overview of the theory of stable random variables can be found in \citet{zolotarev}, \cite{uchaikin} and \citet{samorodnitsky}.

%Many practitioners rely on normal distribution due to its nice property translated into the Central Limit Theorem, which states that a~sum of independent and identically distributed random variables with finite variance has asymptotically normal distribution. The Gaussian model might seem appropriate in many practical applications, where the studied random variable can be considered a~sum of a~large number of random effects. But if the summands have infinite variance, the central limit theorem cannot be applied. In that case the only possible limit of the sum has stable distribution; this result generalizes the central limit theorem and is due to \cite{gnedenko_kolmogorov}. A classical reference to the theory of sums of independent random variables and connected limit theorems is \citet{ibragimov}. 

In many practical applications continuous distributions are often preferred over discrete distributions because they offer more flexibility. There are however cases of practical applications where one need to describe heavy tails in discrete data. Citations of scientific papers (first observed by \cite{price}), word frequency (\cite{zipf}) and  population of cities are all well known examples of discrete data with power tails. A simple discrete power law distribution was introduced by \cite{zipf} and relied on the zeta function (therefore called Zipf or zeta distribution). 

Another possibility is to consider discrete variants of stable and $\nu$-stable distributions. The notion of discrete stability for lattice random variables on non-negative integers was introduced in \cite{steutel}. They introduced so called binomial thinning operator $\odot$ for normalization of discrete random variables. That means that instead of standard normalization  $aX$ by a~constant $a\in (0,1)$, they consider $a\odot X = \sum_{i=1}^X \epsilon_i$, where $\epsilon_i$ are i.i.d.~random variables with Bernoulli distribution with parameter $a$. As opposed to the standard normalization, this thinning operation conserves the integral property of a~discrete random variable $X$. Together with a~study of discrete self-decomposability they obtained the~form of generating function of such discrete stable distributions. By considering only non-negative discrete random variables, they obtained a~discrete version of $\alpha$-stable distributions that are totally skewed to the right. Moreover, the construction allows the index of stability $\alpha$ only smaller or equal to one. \cite{devroye} studied three classes of discrete distributions connected to stable laws, one of them being the discrete stable distribution. \cite{devroye} derived distributional identities for these distributions offering a~method for generating random samples. \cite{christoph} studied discrete stable distributions more into details, offering formulas for the probabilities as well as their asymptotic behaviour. They showed that the discrete stable distribution belongs to the domain of normal attraction of stable distribution totally skewed to the right with index of stability smaller than one. The non-existence of a~closed form formula of the probability mass function and non-existence of moments implies that the classical parameter estimation procedures such as maximum likelihood and method of moments cannot be applied. \cite{marcheselli} and \cite{doray} suggested some methods of parameter estimation of the discrete stable family based on the empirical characteristic function or on the empirical probability generating function.

Discrete stable distributions in limit sense on the set of all integers were introduced in \cite{slamova2012}. Two new classes of discrete distributions were introduced, generalizing the definition of discrete stable distribution of \cite{steutel} on random variables on the set of all integers. It was shown that the newly introduced symmetric discrete stable distribution can be considered a~discrete analogy of symmetric $\alpha$-stable distribution with index of stability $\alpha \in (0,2]$, whereas the introduced discrete stable distribution for random variables on $\Z$ can be viewed as a~discrete analogy of $\alpha$-stable distribution with index of stability $\alpha \in (0,1) \cup \{2\}$ and with skewness $\beta$. \cite{slamovaMME} gave two distributional identities for the symmetric discrete stable and discrete stable random variables, allowing for simple random generator. Possible estimation procedures for the class of discrete stable laws were also considered. 

The aim of this paper is to study different generalizations of the strict stability property with a~particular focus on discrete distributions with some form of stability property. The starting point of the article are discrete stable distributions introduced in \cite{steutel}. Their definition of discrete stability is a~simple generalization of the classical stability property where they consider only one type of thinning operator. The classical stability property can be formulated in several equivalent ways and our aim is to study generalizations of these equivalent definitions for the discrete case. We propose three definitions of discrete stability for random variables on non-negative integers. The main focus is on the first definition that generalizes the definition of \cite{steutel} by allowing the thinning operator to be an arbitrary distribution satisfying certain condition. We introduce also the symmetric and asymmetric variant of discrete stable distribution. The definition of discrete stability on all integers, similarly as in \cite{slamova2012}, is possible only in the limit sense. 

In Chapter 2 three possible definitions of discrete stability for non-negative integer-valued random variables are given. These definitions consider different approaches to introducing discrete stability, each of them being a~discrete version of a different definition of stability in the usual sense. The first definition generalizes the approach taken by \cite{steutel} and considers a~general thinning operator to normalize the sum of discrete random variables. The second definition takes the opposite path and uses a~general so called portlying operator to normalize discrete random variables. The last approach combines the two definitions and as it turns out includes the previous two definitions. Examples of the thinning and portlying operators for which a positive discrete stable random variable exists are provided. Chapters 3 to 6 are dedicated to the study of analytical properties of discrete stable distributions in the first sense. The study is focused mainly on the class of distributions connected to modified geometric thinning operator and we give results on characterizations, probabilities, moments, limiting distributions and asymptotic behaviour for positive and symmetric discrete stable random variables. Section 6 gives also some results on properties of positive discrete stable random variables with Chebyshev thinning operator.

\section{On definitions of discrete stability}\label{ch_definitions}

\cite{slamova2014} introduced a~possible approach to obtain discrete analogies of stable distributions. By approximation of the characteristic function of stable distribution or of its L\'evy measure three discrete distributions were obtained. These distributions are discrete approximations of the stable distributions and it is not clear what properties they share with the stable distributions -- by the construction it is obvious they have the same tail behaviour, but it is not clear whether they share other properties as the stability property, self-similarity, infinite divisibility and others. In this Section we define three new classes of discrete probability distributions by generalizing the stability property for discrete random variables. 

The strict stability property of continuous random variables can be defined in several ways. We say that a~random variable $X$ is strictly stable if one of the following holds
\begin{align}
X &\stackrel{d}{=} a_n \sum_{i=1}^n X_i,\label{eq_def_stability1} \\ 
A_n X &\stackrel{d}{=} \sum_{i=1}^n X_i, \label{eq_def_stability2} \\ 
c X & \stackrel{d}{=} a X_1 + b X_2,  \label{eq_def_stability3}
\end{align}
where $X_1, X_2, \dots, X_n$ are independent copies of $X$ and $a_n$, $A_n$, $a,b$ and $c$ are positive constants. If we want to define a~discrete analogy of stability we have to reconsider the normalization by the constants $a_n$, $A_n$, and $a,b$ and $c$, as the normalized random variables are not necessarily integer-valued. We may consider the following modification. Consider for example the first definition and let us assume that $X$ is non-negative integer-valued random variable. We may write $$X = \underbrace{1 + 1 + \dots + 1}_{X \text{ times}}, \quad \text{and} \quad pX = \underbrace{p + p + \dots + p}_{X \text{ times}},$$ where we normalize $X$ by a~constant $p \in (0,1).$ Instead we can consider a~thinning operator $p \odot X$, where $$p \odot X = \underbrace{\varepsilon_1 + \varepsilon_2 + \dots + \varepsilon_X}_{X \text{ times}},$$ where $\varepsilon_i$ are i.i.d.~Bernoulli random variables with $\E \varepsilon_i = p$, i.e. $$\varepsilon_i = \left \{ \begin{array}{ll} 1, & \text{with probability } p, \\ 0, &  \text{with probability } 1-p. \end{array} \right.$$ 

In the following Sections we introduce three different definitions of discrete stability ge\-ne\-ralizing the definitions of strict stability \eqref{eq_def_stability1} -- \eqref{eq_def_stability3} for the case of non-negative integer-valued random variables. 

%%%%%%%%%%%%%%%%%%%%%%%%%%%%%%%%%%%%%%%%%%%%%%%%%%%%%%%%%%%%%%%%%%%%%%%%%%%%%%%%%%%%%%%%%%%%%%%%%%%%%%%%%%%%%%%%%%%%%%%%%%%%%%%%%%%%%%%%%%%%%%%%%%%%%%%%%%%%%%%%%%%%%%%%%%%%%%%%%%%%%%%%%%%%%
%%%%%%%%%%%%%%%%%%%%%%%%%%%%%%%%%%%%%%%%%%%%%%%%%%%%%%%%%%%%%%%%%%%%%%%%%%%%%%%%%%%%%%%%%%%%%%%%%%%%%%%%%%%%%%%%%%%%%%%%%%%%%%%%%%%%%%%%%%%%%%%%%%%%%%%%%%%%%%%%%%%%%%%%%%%%%%%%%%%%%%%%%%%%%

\subsection{On first definition of discrete stable distributions}\label{section:ds1}

In this Section we give a~definition of discrete stability that generalizes the first definition of strict stability \eqref{eq_def_stability1} for discrete random variables. The multiplication by a~constant $a_n$ can be understood as a~normalization of the sum $\sum X_i$, or normalization of the individual summands $X_i$. In the case of discrete random variables one can not use this normalization as it violates the integral property of the summands. We need to find a~different normalization that maintains the integral property. One possibility is to use the binomial thinning operator $$\tilde{X}(\alpha) = \alpha \odot X = \sum_{i=1}^X \epsilon_i, \quad \text{where} \quad \p(\epsilon_i=1) = 1-\p(\epsilon_i=0) = \alpha,$$ instead of $a_n X_i$. This normalization was used in \cite{steutel} to define discrete stability on $\N_0$. One can generalize this definition of discrete stability by considering a~general normalization, or ``thinning'' operator. 
\begin{definition}\label{def_PDS}
Let $X, X_1, X_2, \dots, X_n, \dots$ denote a~sequence of independent and identically distributed (i.i.d.)~non-negative integer-valued random variables. Assume that for every $n \in \N$ there exists a~constant $p_n \in (0,1)$ such that 
\begin{equation}\label{def_discretestability1}
X \stackrel{d}{=} \sum_{i=1}^n \tilde{X}_i(p_n), \quad \text{where} \quad \tilde{X}_i(p_n) = p_n \odot X_i = \sum_{j=1}^{X_i} \varepsilon_j^{(i)}(p_n),
\end{equation}
and $\varepsilon_j^{(i)}(p_n)$ are i.i.d.~non-negative integer-valued random variables. Then we say that $X$ is \textit{positive discrete stable random variable in the first sense}.
\end{definition}

This definition is rather general as it offers a~flexibility on the choice of the ``thinning'' distribution of random variables $\varepsilon$. This flexibility is however limited as a positive discrete stable random variable exists only for some choice of the thinning distribution. A question is therefore how to describe the family of thinning distributions for which a~positive discrete stable random variables exists. 

Let us denote the probability generating functions of the random variables $X$ and $\varepsilon(p_n)$ by $\Pcal(z) = \E[z^X]$ and $\Qcal_{p_n}(z) = \E[z^{\varepsilon(p_n)}]$ respectively. There is an equivalent definition of positive discrete stability in terms of those probability generating functions.

\begin{proposition}\label{prop_defPDS}
A random variable $X$ is positive discrete stable if and only if for all $n \in \N$ there exists a~constant $p_n \in (0,1)$ such that 
\begin{equation}\label{eq_discretestability1}
\Pcal(z) = \Pcal^n(\Qcal_{p_n}(z)).
\end{equation}
\end{proposition}
\begin{proof}
It follows from the definition \eqref{def_discretestability1} that $X$ is positive discrete stable if and only if
$$\Pcal(z) = \left[\Pcal_{\tilde{X}}(z)\right]^n.$$
The probability generating function of $\tilde{X}$ can be computed in the following way.
\begin{align*}
\Pcal_{\tilde{X}}(z) & = \E\left[z^{\tilde{X}} \right] = \sum_{k=0}^{\infty} \p(X=k) \E\left[z^{\sum_{j=1}^X \varepsilon_j(p_n)} \left| X = k \right. \right]\\
 & = \sum_{k=0}^{\infty} \p(X=k) \left(\E\left[z^{\varepsilon_1(p_n)}\right] \right)^k \\
& = \Pcal(\Qcal_{p_n}(z)).
\end{align*}
Hence $X$ is positive discrete stable if and only if its probability generating function satisfy the relation 
\begin{equation*}
\Pcal(z) = \Pcal^n\big(\Qcal_{p_n}(z)\big).
\end{equation*}
\end{proof}

\begin{remark}
It follows from the definition that a~positive discrete stable random variable $X$ is infinitely divisible: for every $n \in \N$ there exist random variables $Y_1, Y_2, \dots, Y_n$ such that $$X \stackrel{d}{=} Y_1 + Y_2 + \dots + Y_n.$$ This obviously holds for $Y_i = \tilde{X}_i(p_n).$
\end{remark}

Further denote by $\q$ a~semigroup generated by the family of probability generating functions $\{\Qcal(z) = \Qcal_{p_n}(z), n \in \N\}$ with operation of superposition $\circ$. It can be shown that a~superposition of two probability generating functions is again a~probability generating function.
\begin{lemma}
If $\Qcal_1(z)$ and $\Qcal_2(z)$ are two probability generating functions of two random variables with values in $\N_0$, then their superposition $$\Qcal_1 \circ \Qcal_2(z):= \Qcal_1(\Qcal_2(z))$$ is also a~probability generating function of some random variable with values in $\N_0$.
\end{lemma}
\begin{proof}
Let $N$ be a~random variable with values in $\N_0$ with probability generating function $\Qcal_1$ and $X_1,X_2,\dots$ i.i.d. random variables with values in $\N_0$ with probability generating function $\Qcal_2$. Define a~new random variable $S$ by $$S = \sum_{i=1}^N X_i.$$ Then $S$ is a~random variable with values in $\N_0$. Its probability generating function can be computed using the fundamental formula of conditional expectation as follows
\begin{align*}
\Qcal_S(z) & = \E\left[z^S\right] = \E\left[z^{\sum_{i=1}^N X_i}\right] = \sum_{n=0}^{\infty} \p(N=n) \E\left[z^{\sum_{i=1}^N X_i} | N = n\right] \\
& = \sum_{n=0}^{\infty} \p(N=n) \E\left[z^{\sum_{i=1}^n X_i}\right] = \sum_{n=0}^{\infty} \p(N=n)  \left[\E z^{X_1}\right]^n \\
& = \sum_{n=0}^{\infty} \p(N=n)  \left[\Qcal_2(z)\right]^n = \Qcal_1(\Qcal_2(z)).
\end{align*}
So the superposition $\Qcal_1 \circ \Qcal_2(z)$ is a~probability generating function of random variable $S$ with values in $\N_0$.
\end{proof}

Now we show that the semigroup $\q$ must be commutative.

\begin{theorem}\label{Qcommutative}
Let $X$ be a positive discrete stable random variable. Then the semigroup $\q$ must be commutative. 
\end{theorem}
\begin{proof}
Let us denote $G(z) = \log \Pcal(z)$. Then \eqref{eq_discretestability1} is equivalent to 
\begin{equation}\label{eq_discretestability2}
G(z) = n G(\Qcal_{p_n}(z)), \quad n \in \N.
\end{equation}
Let $G(z)$ be a~solution of \eqref{eq_discretestability2}. Then for all $n \in \N$ it must hold $$\Qcal_{p_n}(z) = G^{-1}\left(\frac{1}{n}G(z)\right).$$ It follows from here that for all $n_1,n_2 \in \N$
\begin{align*}
\Qcal_{p_{n_1}}\left(\Qcal_{p_{n_2}}(z)\right) & = G^{-1}\left(\frac{1}{n_1} G\left(G^{-1}\left(\frac{1}{n_2}G(z)\right)\right)\right) \\
& = G^{-1}\left(\frac{1}{n_1}\frac{1}{n_2}G(z)\right) \\
& = \Qcal_{p_{n_2}}\left(\Qcal_{p_{n_1}}(z) \right), 
\end{align*}
which means that $\q$ is commutative.
\end{proof}

Similarly as for the classical stable distribution, we can show that the constants $p_n$ have to take form $p_n= n^{-1/\gamma}$ for some $\gamma > 0$.

\begin{theorem}\label{th:indexofstability}
Let $X$ be a~positive discrete stable random variable in the first sense. Then there exists $\gamma > 0$ such that $p_n$ in \eqref{def_discretestability1} takes form $$p_n = n^{-1/\gamma}.$$
\end{theorem}
\begin{proof}
The proof follows \cite[\S 2.4]{uchaikin} where a~similar statement for stable distributions is proved. From the definition it follows that for every $n \geq 2$ we have $X \stackrel{d}{=} \sum_{i=1}^n \tilde{X}_i(p_n)$ where $X_1, X_2, \dots$ are independent copies of $X$. Then 
\begin{align*}
X & \stackrel{d}{=} p_2 \odot X_1 + p_2 \odot X_2, \\
\intertext{therefore also }
X & \stackrel{d}{=} p_2 \odot (p_2 \odot X_1 + p_2 \odot X_2) + p_2 \odot (p_2 \odot X_3 + p_2 \odot X_4).
\end{align*}
But the operation $\odot$ is associative: $p \odot (p \odot X) = p^2 \odot X$. Let us denote $Y = p_2 \odot X_1 + p_2 \odot X_2$. Then (using result from proof of Proposition \ref{prop_defPDS}) $$\Pcal_{p_2 \odot Y}(z) = \Pcal_Y(\Qcal_{p_2}(z)) = \Pcal^2(\Qcal_{p_2}(\Qcal_{p_2}(z))) = \Pcal^2(\Qcal_{p_2^2}(z)),$$ because $\q$ is commutative. Therefore
\begin{align}
X & \stackrel{d}{=} p_2^2 \odot  X_1 + p_2^2 \odot X_2 + p_2^2 \odot X_3 + p_2^2 \odot X_4 \notag \\
\intertext{and similarly for every $n = 2^k$}
X & \stackrel{d}{=} p_2^k \odot  X_1 + p_2^k \odot X_2 + \dots + p_2^k \odot X_n. \label{p2k} \\
\intertext{On the other hand, we have}
X & \stackrel{d}{=} p_n \odot  X_1 + p_n \odot X_2 + \dots + p_n \odot X_n. \label{pn}
\end{align}
Comparing \eqref{p2k} with \eqref{pn}, with $n=2^k$, we have $p_n = p_2^k.$ Hence $$\log p_n = k \log p_2 = \frac{\log n}{\log 2} \log p_2 = \log n^{\log p_2 / \log 2}.$$ So we obtain that $$p_n = n^{-1/\gamma_2}, \quad \gamma_2 = -\log 2 / \log p_2 > 0, \quad n = 2^k, k=1,2, \dots.$$ 
In a~similar way, starting with sums with 3 terms $X \stackrel{d}{=} p_3 \odot X_1 + p_3 \odot X_2 + p_3 \odot X_3,$ we get 
$$p_n = n^{-1/\gamma_3}, \quad \gamma_3 = -\log 3 / \log p_3 > 0, \quad n = 3^k, k=1,2, \dots.$$
And in general case, $$p_n = n^{-1/\gamma_m}, \quad \gamma_m = -\log m / \log p_m > 0, \quad n = m^k, k=1,2, \dots.$$
But for $m=4$ we obtain both $\gamma_4 = - \log 4/\log p_4 $ and $\log p_4 = -1/\gamma_2 \log 4.$ Hence $\gamma_4 = \gamma_2$. By induction we conclude that $\gamma_m = \gamma$ for all $m$ and therefore $$p_n = n^{-1/\gamma}, \quad \text{for all} \quad n \geq 2.$$
\end{proof}

The question is how to extend the definition of discrete stability to contain not only random variables on $\N_0$, but also on the whole integers $\Z$. It is obvious that the sum in definition of $\tilde{X}$ does not make sense for random variables that can achieve negative values. One possibility is to take the positive and negative part of $X$ separately and consider again the same thinning operator. We can, however, obtain a~wider class of distributions if we assume a~different thinning operator than in Definition \ref{def_PDS}. 

\begin{definition}\label{def_DS}
Let $X, X_1, X_2, \dots, X_n, \dots$ denote a~sequence of independent and identically distributed (i.i.d.)~integer-valued random variables. Assume that for every $n \in \N$ there exists a~constant $p_n \in (0,1)$ such that 
\begin{equation}\label{def_discretestabilityZ}
X \stackrel{d}{=} \lim_{n\to \infty}\sum_{i=1}^n \bar{X}_i(p_n), \quad \text{where} \quad \bar{X}_i(p_n) =  \sum_{j=1}^{X_i^+} \varepsilon_j^{(i)}(p_n) - \sum_{j=1}^{X_i^-} \epsilon_j^{(i)}(p_n),
\end{equation}
$\varepsilon_j^{(i)}(p_n),\epsilon_j^{(i)}(p_n)$ are i.i.d.~integer-valued random variables, and $X^+$ and $X^-$ are the positive and negative part of $X$, respectively (i.e. $X^+ = X$ if $X \geq 0$ and 0 otherwise, $X^- = -X$ if $X<0$ and 0 otherwise).  Then we say that $X$ is \textit{discrete stable random variable in the limit sense}.
\end{definition}

The main difference is that we do not assume the random variables $\varepsilon, \epsilon$ to be non-negative. The definition of discrete stability is only in the limit sense, not the algebraic one where we have equivalence in distribution in \eqref{def_discretestabilityZ} instead of the limit. 

Let us denote again the probability generating function of the random variables $X$ and $\varepsilon(p_n))$ (and also $\epsilon(p_n)$) by $\Pcal(z) = \E[z^X]$ and $\Rcal_{p_n}(z) = \E[z^{\varepsilon(p_n)}]= \E[z^{\epsilon(p_n)}]$ respectively. We denote by $\Pcal_1$ the generating function of the sequence $\{a_k = \p(X=k), k=1,2,\dots\}$ and by $\Pcal_2$ the generating function of the sequence $\{b_k = \p(X=k), k=-1,-2,\dots\}$. We denote $\Pcal_0 = \p( X=0)$. It is obvious that the generating function of $X^+$ is $\Pcal_0 + \Pcal_1(z)$, and the generating function of $X^-$ is $\Pcal_2(z)$. There is an equivalent definition of discrete stability in the limit sense in terms of those generating functions.

\begin{proposition}\label{prop_defDS2}
A random variable $X$ is discrete stable in the limit sense if and only if for all $n \in \N$ there exists a~constant $p_n \in (0,1)$ such that 
\begin{equation}\label{eq_discretestabilityZ2}
\Pcal(z) = \lim_{n \to \infty} \left[\Pcal_0 + \Pcal_1(\Rcal_{p_n}(z)) + \Pcal_2\left(\Rcal_{p_n}(1/z)\right)\right]^n.
\end{equation}
\end{proposition}
\begin{proof}
It follows from the definition \eqref{def_discretestabilityZ} that $X$ is discrete stable if and only if
$$\Pcal(z) = \lim_{n\to \infty} \left[\Pcal_{\bar{X}}(z)\right]^n.$$
The probability generating function of $\bar{X}$ can be computed in the following way.
\begin{align*}
\Pcal_{\bar{X}}(z) & = \E\left[z^{\bar{X}} \right] = \sum_{k=-\infty}^{\infty} \p(X=k) \E\left[z^{\sum_{j=1}^{X^+} \varepsilon_j(p_n) - \sum_{j=1}^{X^-} \epsilon_j(p_n) } \left| X = k \right. \right]\\
 & = \sum_{k=0}^{\infty} \p(X=k) \left(\E\left[z^{\varepsilon_1(p_n)}\right] \right)^k + \sum_{k=-\infty}^{-1}  \p(X=k) \left(\E\left[z^{-\epsilon_1(p_n)}\right]\right)^{-k} \\
& = \Pcal_0 + \Pcal_1(\Rcal_{p_n}(z)) + \Pcal_2(\Rcal_{p_n}(1/z)).
\end{align*}
Hence $X$ is discrete stable if and only if its probability generating function satisfies the relation 
\begin{equation*}
\Pcal(z) = \lim_{n \to \infty} \left[\Pcal_0 + \Pcal_1(\Rcal_{p_n}(z)) + \Pcal_2(\Rcal_{p_n}(1/z))\right]^n.
\end{equation*}
\end{proof}

It is important to note that we do not define discrete stability property in the algebraic sense as we defined it for the non-negative integer-valued random variables. This also leads to the fact that we have no condition on the thinning operator $\Rcal$ similar to Theorem \ref{Qcommutative}.

In the following Subsections we introduce some examples of commutative semigroups $\q$ leading to different positive discrete stable random variables. We will also give corresponding examples of discrete stable distributions in the limit sense. The proofs of the results will be provided in Chapter 5.

%%%%%%%%%%%%%%%%%%%%%%%%%%%%%%%%%%%%%%%%%%%%%%%%%%%%%%%%%%%%%%%%%%%%%%%%%%%%%%%%%%%%%%%%%%%%%%%%%%%%%%%%%%%%%%%%%%%%%%%%%%%%%%%%%%%%%%%%%%%%%%%%%%%%%%%%%%%%%%%%%%%%%%%%%%%%%%%%%%%%%%%%%%
%%%%%%%%%%%%%%%%%%%%%%%%%%%%%%%%%%%%%%%%%%%%%%%%%%%%%%%%%%%%%%%%%%%%%%%%%%%%%%%%%%%%%%%%%%%%%%%%%%%%%%%%%%%%%%%%%%%%%%%%%%%%%%%%%%%%%%%%%%%%%%%%%%%%%%%%%%%%%%%%%%%%%%%%%%%%%%%%%%%%%%%%%%

\paragraph{Binomial thinning operator.}
Assume that the probability generating function $\Qcal$ is that of Bernoulli distribution with parameter $p \in (0,1)$, i.e.~we have $\Qcal(z) = pz + (1-p).$ It is easy to verify that the semigroup $\q$ generated by probability generating functions of this form is commutative, as $$\Qcal_{p_1}\left(\Qcal_{p_2}(z)\right) = p_1 p_2 z + (1-p_1 p_2).$$ This operator was used in \cite{steutel} to define discrete stable distribution on $\N_0$ and it was showed there that it leads to a~distribution with probability generating function given by 
\begin{equation}\label{pz_simplepositivediscretestable}
\Pcal(z) = \exp\left\{-\lambda (1-z)^{\gamma}\right\}, \quad \gamma \in (0,1], \; \lambda > 0 . 
\end{equation}

To obtain a generalization of this distribution on $\Z$ we can consider two-sided binomial thinning operator defined as $\Rcal(z) = (1-p)+ p q z+ p(1-q)z^{-1}$, where $q \in [0,1]$. This thinning operator leads to a~distribution on $\Z$ with probability generating function given by 
\begin{align*}\label{pz_simplediscretestable}
\Pcal(z) & = \exp\left\{-\lambda \left(\frac{1+\beta}{2}\right) \left(1-q z - (1-q) \frac{1}{z} \right)^{\gamma}- \lambda \left(\frac{1-\beta}{2}\right) \left(1-q \frac{1}{z} - (1-q) z\right)^{\gamma}\right\},  
\end{align*}
with $  \lambda > 0, \gamma \in (0,1], \beta \in [-1,1], q \in [0,1]$. We can see that for $\beta = 1$ and $q=1$ the distribution reduces to positive discrete stable \eqref{pz_simplepositivediscretestable}. 

%%%%%%%%%%%%%%%%%%%%%%%%%%%%%%%%%%%%%%%%%%%%%%%%%%%%%%%%%%%%%%%%%%%%%%%%%%%%%%%%%%%%%%%%%%%%%%%%%%%%%%%%%%%%%%%%%%%%%%%%%%%%%%%%%%%%%%%%%%%%%%%%%%%%%%%%%%%%%%%%%%%%%%%%%%%%%%%%%%%%%%%%%%
%%%%%%%%%%%%%%%%%%%%%%%%%%%%%%%%%%%%%%%%%%%%%%%%%%%%%%%%%%%%%%%%%%%%%%%%%%%%%%%%%%%%%%%%%%%%%%%%%%%%%%%%%%%%%%%%%%%%%%%%%%%%%%%%%%%%%%%%%%%%%%%%%%%%%%%%%%%%%%%%%%%%%%%%%%%%%%%%%%%%%%%%%%

\paragraph{Thinning operator of geometric type.}\label{def_mgo}
A generalization of the previous example can be obtained if we consider $\Qcal$ to be the probability generating function of modified geometric distribution with parameters $p \in (0,1)$ and $\kappa \in [0,1)$. Consider a~function
\begin{equation}\label{Qz}
\Qcal(z) = \left( \frac{(1-p)+(p-\kappa)z^m}{(1-p\kappa)-\kappa(1-p)z^m} \right)^{\frac{1}{m}},\quad 
\begin{array}{l}
\{0 \leq \kappa < 1, \; \, 0< p < 1, \; \, m =1\} \\
\text{ or }  \\
\{0 < p < \kappa < 1, \; \,  \; \, m \in \N, \; \, m > 1\}.
\end{array}
\end{equation}

\begin{lemma}
The function $\Qcal(z)$ is a~probability generating function.
\end{lemma}
\begin{proof}
To verify that $\Qcal(z) = \sum_{n=0}^{\infty} q_n z^n$ is a~probability generating function we have to show that the generating sequence $\{q_n,n=0,1,\dots\}$ is a~probability mass function, i.e.~$\sum_n q_n = 1$, and $0\leq q_n \leq 1$. We see that  $\sum_n q_n = \Qcal(1) = 1$. We expand $\Qcal$ into a~power series to obtain the generating series $\{q_n,n=0,1,\dots\}$. We will treat the case of $m=1$ and $m>1$ separately.\\

Let first $m=1$. Then we obtain $\Qcal(z) = \sum_{n=0}^{\infty}q_n z^{n}, $ with 
\begin{align*}
q_0 & = \frac{1-p}{1-p\kappa}, \\
q_n & = p~\kappa^{n-1} \frac{(1-p)^{n-1}(1-\kappa)^2}{(1-p\kappa)^{n+1}}, \quad n \geq 1.
\end{align*}
We can easily verify that for $0<\kappa< 1$ and $0 < p < 1$, $\{q_n\}$ is a~probability mass function and thus $\Qcal$ is a~probability generating function.\\

Let $m>1$. We obtain $$\Qcal(z) = \sum_{n=0}^{\infty}q_n z^{m n}, $$ where the coefficients $q_n$ are given as
\begin{align*}
q_{n} & = \sum_{j=0}^n \left(\frac{1-p}{1-p \kappa}\right)^{1/m+n-j} \left(\frac{p-\kappa}{1-p}\right)^j \kappa^{n-j}\binom{1/m+n-j-1}{n-j}\binom{1/m}{j}, \quad n \in \N_0.
\end{align*}
This can be reduced to $$q_n = \kappa^n \left(\frac{1 - p}{1-p\kappa}\right)^{1/m + n}  \binom{1/m + n - 1}{n} \mathrm{{}_2F_1}\left(\{-1/m, -n\},1-1/m-n, \frac{(\kappa-p)(1 - p \kappa)}{(1 - p)^2 \kappa}\right).$$ It follows from the properties of the hypergeometric ${}_2\mathrm{F}_1$ function that $0 \leq q_n \leq 1$ if and only if $$ 0\leq \frac{(\kappa-p)(1 - p \kappa)}{(1 - p)^2 \kappa} \leq 1.$$ This is fulfilled if and only if $0 \leq p \leq \kappa \leq 1$. However, if $k=1$ or $p= k$ or $p=0$ we obtain a~degenerate distribution. From here if follows that $\{q_n\}$ is a~probability mass function if and only if $0<p<\kappa<1$. 
\end{proof}

The distribution given by the probability generation function $\Qcal$ with $m=1$ is sometimes called modified geometric distribution (\cite{phillips}) or zero-modified geometric distribution (\cite{johnson}). This distribution is obtained as a mixture of a degenerate distribution and geometric distribution: let $U$ be a degenerate random variable identically equal to zero, and let $V$ be a geometrically distributed random variable with parameter $b \in (0,1]$. Let $q \in (0,1)$ and denote $Z = qU + (1-q)V$. Then the probability generating function of the mixture $Z$ is given as
$$\Qcal(z) = q + (1-q) \frac{bz}{1-(1-b)z}.$$ We can reparametrize this distribution, by putting $$q = \frac{1-p}{1-p\kappa}\quad \text{and} \quad b = \frac{1-\kappa}{1-p\kappa}$$ with $p \in (0,1)$ and $\kappa \in [0,1)$. Then the probability generating function takes form \eqref{Qz} with $m=1$. 

The parameter $m$ specifies the lattice of the distribution. We will denote the distribution with probability generating function $\Qcal$ by $\mathcal{G}(p,\kappa,m)$. If $m=1$, we will write simply $\mathcal{G}(p,\kappa)$.

\begin{lemma}
The function $\Qcal(z)$ can be decomposed as 
\begin{equation}\label{G_decom}
\Qcal(z) = S^{-1} \circ B_p \circ S(z), \quad \text{where} \quad  S(z) = \frac{(1-\kappa)z^m}{1-\kappa z^m}, \quad B_p(z) = pz + 1-p.
\end{equation}
\end{lemma}
\begin{proof}
The decomposition can be verified by computation, as $$S^{-1}(y)= \left(\frac{y}{(1-\kappa)+\kappa y}\right)^{1/m}.$$ 
\end{proof}

The function $B_p(z)$ is the probability generating function of the Bernoulli distribution. In previous Subsection we showed, that $B_p$ generates a~commutative semigroup. Using the decomposition \eqref{G_decom} it is easy to see that the semigroup $\q$ is commutative, as
\begin{align*}
\Qcal_{p_1}\left(\Qcal_{p_2}(z)\right) & = S^{-1} \circ B_{p_1} \circ S \circ S^{-1} \circ B_{p_2} \circ S(z) \\
& = S^{-1} \circ B_{p_1} \circ B_{p_2} \circ S(z) 
\end{align*}
and we already showed that $B_{p_1} \circ B_{p_2}(z) = B_{p_1 p_2}(z)$.

If we choose $m=1$ and $\kappa = 0$ the modified geometric distribution reduces to the Bernoulli distribution. We can modify the operator $\Qcal$ and consider two-sided thinning operator of geometric type. This can be done by considering $$\Rcal(z) = S^{-1} \circ B_p \circ S^{(2)}(z) \quad \text{ where }\quad S^{(2)}(z) = q S(z) + (1-q) S(z^{-1})\, \quad q \in [0,1]$$ instead of $\Qcal(z)$. We will denote two-sided modified geometric distribution by $2\mathcal{G}(p,\kappa,q,m)$. We see that $\Qcal$ is obtained from $\Rcal$ by considering $q = 1$.

We will study discrete stable distributions with $\mathcal{G}$ thinning operator (of geometric type) more into details in Sections 3--5. It will be shown there that this choice of thinning operator in Definition \ref{def_PDS} leads to a~distribution with probability generating function given by
\begin{equation}\label{pz_positivediscretestable}
\Pcal(z) = \exp \left\{-\lambda \left(\frac{1-z^m}{1-\kappa z^m}\right)^{\gamma} \right\}, \quad \lambda > 0, \; \gamma \in (0,1], \; \kappa \in [0,1), \; m\in \N. 
\end{equation}

%%%%%%%%%%%%%%%%%%%%%%%%%%%%%%%%%%%%%%%%%%%%%%%%%%%%%%%%%%%%%%%%%%%%%%%%%%%%%%%%%%%%%%%%%%%%%%%%%%%%%%%%%%%%%%%%%%%%%%%%%%%%%%%%%%%%%%%%%%%%%%%%%%%%%%%%%%%%%%%%%%%%%%%%%%%%%%%%%%%%%%%%%%%%%
%%%%%%%%%%%%%%%%%%%%%%%%%%%%%%%%%%%%%%%%%%%%%%%%%%%%%%%%%%%%%%%%%%%%%%%%%%%%%%%%%%%%%%%%%%%%%%%%%%%%%%%%%%%%%%%%%%%%%%%%%%%%%%%%%%%%%%%%%%%%%%%%%%%%%%%%%%%%%%%%%%%%%%%%%%%%%%%%%%%%%%%%%%%%%

\paragraph{Thinning operator of Chebyshev type.}
Let us consider a~function of the following form
\begin{equation}\label{Q_chebyshev}
\Qcal(z) = \frac{2\left(b+T_p\left(\frac{(1+b)z-2b}{2-(1+b)z}\right)\right)}{(1+b)\left(1+T_p\left(\frac{(1+b)z-2b}{2-(1+b)z}\right)\right)},
\end{equation}
where $p \in (0,1)$ and $b \in (-1,1)$ and $T_p(x) = \cos\left(p \arccos x \right).$

\begin{remark}
The function $T_n$ for $n \in \N$ is called Chebyshev polynomial. It belongs to the class of orthogonal polynomials. There is an extensive literature about Chebyshev polynomials, see for example \cite{rivlin}. Chebyshev polynomials are commutative, $T_n \circ T_m (x) = T_m \circ T_n (x)$; they have the nesting property, $T_n \circ T_m (x) = T_{mn}(x)$. This holds true also for $T_p(x)$ with $p \in (0,1)$, defined as $T_p(x) = \cos\left(p \arccos x \right)$. However, in this case $T_p$ is not a~polynomial any more. 
\end{remark}

The function $\Qcal(z)$ can be decomposed in the following way:
\begin{equation}\label{QCheb_dec}
\Qcal(z) = R^{-1} \circ T_p \ circ R(z), \quad \text{where} \quad R(z) = \frac{(1+b)z-2b}{2-(1+b)z}, R^{-1}(y) = \frac{2(b+y)}{(1+b)(1+y)}.
\end{equation}

\begin{lemma}
The function $\Qcal(z)$ is a~probability generating function.
\end{lemma}
\begin{proof}
Let us consider only the case of $b=0$ and $p = \tfrac{1}{n}$, $n \in \N, n\geq 2$. Then we can rewrite the function $\Qcal(z)$ in the following form
\begin{align*}
\Qcal(z) = \frac{2 \cos\left(\frac{1}{n} \arccos \frac{z}{2-z}\right)}{1 + \cos\left(\frac{1}{n} \arccos \frac{z}{2-z}\right)}.
\end{align*}
Using the exponential and logarithmic forms of $\cos$ and $\arccos$ functions $\cos(x) = (e^{\im x}+e^{-\im x})/2$ and $\arccos(x) = \tfrac{\pi}{2} + \im \log\left(\im x + \sqrt{1-x^2}\right)$ we can rewrite $\cos(\tfrac{1}{n}\arccos y)$ into the following form
\begin{align*}
\cos(\tfrac{1}{n}\arccos y) & = \tfrac{1}{2}\left(e^{\im \frac{\pi}{2n} - \log\left(\im y + \sqrt{1-y^2}\right)^{1/n}} + e^{-\im \frac{\pi}{2n} + \log\left(\im y + \sqrt{1-y^2}\right)^{1/n}}  \right) \\
& = \tfrac{1}{2} e^{\im \frac{\pi}{2n}} \left(\im y + \sqrt{1-y^2}\right)^{-1/n} +  \tfrac{1}{2} e^{-\im \frac{\pi}{2n}}\left(\im y + \sqrt{1-y^2}\right)^{1/n} \\
& = \tfrac{1}{2} \frac{e^{\im \frac{\pi}{n}} + (\im y + \sqrt{1-y^2})^{2/n}}{e^{\im \frac{\pi}{2n}} (\im y + \sqrt{1-y^2})^{1/n}}.
\end{align*}
Hence $\Qcal(z)$ simplifies into (we use substitution $y = \tfrac{z}{2-z}$)
\begin{align*}
\Qcal(z) & = 2 \frac{e^{\im \frac{\pi}{n}} + (\im y + \sqrt{1-y^2})^{2/n}}{\left[e^{\im \frac{\pi}{2n}} + (\im y + \sqrt{1-y^2})^{1/n}\right]^2} \\
& = 2 \frac{1 + (y - \im \sqrt{1-y^2})^{2/n}}{\left[1 + (y - \im \sqrt{1-y^2})^{1/n}\right]^2}\\
& = \frac{2}{1+ \frac{2}{(y - \im \sqrt{1-y^2})^{1/n}+ (y - \im \sqrt{1-y^2})^{-1/n}}}.
\end{align*}
So for $z \in (0,1]$ we have
\begin{equation*}\label{chebyshevQ}
\Qcal(z) = \frac{2}{1+ \frac{2}{\left(\frac{z}{2-z} - 2 \im \frac{\sqrt{1-z}}{2-z}\right)^{1/n}+ \left(\frac{z}{2-z} - 2 \im \frac{\sqrt{1-z}}{2-z}\right)^{-1/n}}}.
\end{equation*}
We have to show that $\Qcal(z)$ is a~real function of $z$. Let $x = \frac{z}{2-z}$, $y = - 2 \frac{\sqrt{1-z}}{2-z}$ and $u = x + \im y = r (\cos \phi + \im \sin \phi)$. Then using Moivre's formula
\begin{multline*}
\left(\frac{z}{2-z} - 2 \im \frac{\sqrt{1-z}}{2-z}\right)^{1/n}+ \left(\frac{z}{2-z} - 2 \im \frac{\sqrt{1-z}}{2-z}\right)^{-1/n} \\ = r^{1/n} (\cos(\phi/n) + \im \sin(\phi/n)) +  r^{-1/n} (\cos(\phi/n) - \im \sin(\phi/n)).
\end{multline*} This number is real if and only if $r = 1$. But $$r = ||x+\im y|| = \sqrt{x^2 + y^2}  = \frac{z^2 + 4(1-z)}{(2-z)^2} = 1.$$ We conclude that for $z \in (0,1]$ the function $\Qcal(z)$ is real valued. Moreover $\Qcal(1) = 1$. To complete the proof we need to show that $\Qcal(z)$ is a~power series with nonnegative coefficients expressing probabilities.

We denote $\Qcal$ related to the parameter $p$ by $\Qcal_p(z)$.  The inverse function of $\Qcal_p(z)$ is $$\Qcal_p^{-1}(y) = \frac{2T_n\left(\frac{y}{2-y}\right)}{1+T_n\left(\frac{y}{2-y}\right)}.$$ This follows from the decomposition \eqref{QCheb_dec}, $\Qcal_p(z) = R^{-1} \circ T_p \circ R(z)$, where $R(z) = \frac{z}{2-z}$ and from the fact that the inverse function of $T_p(x)$ is $T_n(x)$. This can be verified easily from the definition $T_p(x) = \cos\left(p \arccos x \right).$ For $n \in \N$ is $T_n$ the Chebyshev polynomial.

Consider first the simple case of $n=2$. We know that $T_2(x) = 2x^2 - 1$ (see, for example, \cite{rivlin}). Therefore $$\Qcal_p^{-1}(y) = 1 + \frac{4}{y} - \frac{4}{y^2}.$$ We may inverse this function again to obtain $$\Qcal_p(z) = \Qcal_{1/2}(z) = \frac{-2+2\sqrt{2-z}}{1-z},$$ for $z<1$. The power series expansion is now easy to obtain 
$$\Qcal_{1/2}(z) = \sum_{m=0}^{\infty}  \frac{\sqrt{2}}{2^{m+1}} (-1)^m \binom{\tfrac{1}{2}}{m+1} {}_2F_1\left(1,\tfrac{1}{2}+m, 2+m, \tfrac{1}{2}\right) z^m.$$ It can be verified that the coefficients $$p_m =\frac{\sqrt{2}}{2^{m+1}} (-1)^m \binom{\tfrac{1}{2}}{m+1} {}_2F_1\left(1,\tfrac{1}{2}+m, 2+m, \tfrac{1}{2}\right)$$ are all positive as $\binom{\frac{1}{2}}{m+1}$ is positive for $m$ even and negative for $m$ odd and the hypergeometric function ${}_2F_1(1,\frac{1}{2}+m, 2+m, \frac{1}{2})$ is always positive for $m \geq 0$. Therefore $\Qcal_{1/2}(z)$ is a~probability generating function.\\

Now we will show by induction that $\Qcal_p(z)$ is a~probability generating function for all $p$ of the form $p = 1/2^k$, with $k \in \N$. We already showed that it is true for $p = \tfrac{1}{2}$. Let us assume $\Qcal_{p}(z)$ is a~probability generating function for $p = \tfrac{1}{2^k}$, $k \geq 1$. Because of the nesting property of $T_p$ we have $T_{p/2} = T_p \circ T_{1/2}$, therefore we may write 

\begin{align*}
\Qcal_{p/2}(z) & = R^{-1} \circ T_{p/2} \circ R(z) = R^{-1} \circ T_p \circ T_{1/2} \circ R(z) = \\
& =  R^{-1} \circ T_p \circ R \circ R^{-1} \circ T_{1/2} \circ R(z) \\
& = \Qcal_p \circ \Qcal_{1/2}(z)
\end{align*}

By induction assumption $\Qcal_p(z)$ is a~probability generating function, as well as $\Qcal_{1/2}(z)$. The composition of two probability generating function is a~probability generating function itself, therefore we conclude that $\Qcal_{p/2}(z)$ is probability generating function.
\end{proof}

We denote the probability distribution given by the probability generating function \eqref{Q_chebyshev} by $\mathcal{T}(p,b)$. 

\begin{proposition}
Let $\varepsilon \sim \mathcal{T}(p,b)$. Then $\E \varepsilon = p^2.$
\end{proposition}
\begin{proof}
We compute the expectation of $\varepsilon$ using the property of probability generating functions as $\E \varepsilon = \Qcal'(1).$ By deriving $\Qcal(z)$ we obtain
\begin{align*}
\Qcal'(z) & = \frac{2(1-b)\frac{\dd}{\dd z}T_p(u(z))}{(1+b)(1+T_p(u(z)))^2}, \\
\intertext{where}
u(z) & = \frac{(1+b)z-2b}{2-(1+b)z},\\
\frac{\dd}{\dd z}T_p(u(z)) & = \frac{\dd}{\dd u}T_p(u) u'(z),\\
u'(z) & = \frac{2(1+b)(1-b)}{(2-(1+b)z)^2}.
\intertext{Using the relation between Chebyshev polynomials of the first and second kind (see \cite{erdelyiII}) we obtain}
\frac{\dd}{\dd u}T_p(u) & = p U_{p-1}(u) = p \frac{\sin(p\arccos u)}{\sin(\arccos u)}.
\end{align*}
Putting all together and setting $z=1$, $u = u(1) = 1$ we obtain 
$$\Qcal'(1) = \frac{4(1-b)p^2 \frac{1+b}{1-b}}{4(1+b)} = p^2.$$
\end{proof}

The semigroup $\q$ generated by probability generating functions of this form is commutative. From the decomposition \eqref{QCheb_dec} follows tat
\begin{align*}
\Qcal_{p_1}\left(Q_{p_2}(z)\right)& = R^{-1} \circ T_{p_1} \circ R \circ R^{-1} \circ T_{p_2} \circ R(z) \\
& = R^{-1} \circ T_{p_1} \circ T_{p_2} \circ R(z).
\intertext{But}
T_{p_1} \circ T_{p_2}(x) & = \cos\left(p_1 \arccos\left( \cos \left(p_2 \arccos x\right)\right)\right) \\
& = \cos\left(p_1 p_2 \arccos x\right) \\
& = T_{p_2} \circ T_{p_1}(x).
\end{align*}
We will study discrete stable distributions with Chebyshev type ($\mathcal{T}$) thinning operator more into details in Section \ref{sec:pdsT}. It will be shown there that this choice of thinning operator in Definition \ref{def_PDS} leads to a~distribution with probability generating function given by
\begin{equation}\label{pz_modifiedpositivediscretestable}
\Pcal(z) = \exp \left\{-\lambda \left(\arccos \frac{(1+b)z - 2b}{2-(1+b)z}\right)^{\gamma} \right\}, \quad \gamma \in (0,2], \; \lambda > 0, \;  b \in (-1,1). 
\end{equation}

%%%%%%%%%%%%%%%%%%%%%%%%%%%%%%%%%%%%%%%%%%%%%%%%%%%%%%%%%%%%%%%%%%%%%%%%%%%%%%%%%%%%%%%%%%%%%%%%%%%%%%%%%%%%%%%%%%%%%%%%%%%%%%%%%%%%%%%%%%%%%%%%%%%%%%%%%%%%%%%%%%%%%%%%%%%%%%%%%%%%%%%%%%
%%%%%%%%%%%%%%%%%%%%%%%%%%%%%%%%%%%%%%%%%%%%%%%%%%%%%%%%%%%%%%%%%%%%%%%%%%%%%%%%%%%%%%%%%%%%%%%%%%%%%%%%%%%%%%%%%%%%%%%%%%%%%%%%%%%%%%%%%%%%%%%%%%%%%%%%%%%%%%%%%%%%%%%%%%%%%%%%%%%%%%%%%%

\subsection{On second definition of discrete stable distributions}
In this Subsection we give a~definition of discrete stability that generalizes the second definition of strict stability \eqref{eq_def_stability2} for discrete random variables. The constant $A_n$ in \eqref{eq_def_stability2} takes form $A_n = n^{1/\alpha}$ for some $0<\alpha \leq 2$. Hence the product $A_n X$ is generally not integer-valued and we have to find a~different normalization. Compared to the normalization used in previous Subsection we need a~``portlying'' normalization rather than thinning, therefore we will look for distributions with expected value bigger than 1. 

\begin{definition}\label{def_PDS2}
Let $X, X_1, X_2, \dots, X_n, \dots$ denote a~sequence of independent and identically distributed non-negative integer-valued random variables. Assume that for every $n \in \N$ there exists a~constant $p_n > 0$ such that 
\begin{equation}\label{def_discretestability2}
\hat{X}(p_n) \stackrel{d}{=} \sum_{i=1}^n X_i, \quad \text{where} \quad \hat{X}(p_n) = \sum_{j=1}^{X} \varepsilon_j(p_n),
\end{equation}
and $\varepsilon_j(p_n)$ are i.i.d.~non-negative integer-valued random variables. Then we say that $X$ is \textit{positive discrete stable random variable in the second sense}.
\end{definition}

Let us denote the probability generating functions of the random variables $X$ and $\varepsilon(p_n)$ by $\Pcal(z) = \E[z^X]$ and $\Qcal_{p_n}(z) = \E[z^{\varepsilon(p_n)}]$ respectively. There is an equivalent definition of positive discrete stability in the second sense in terms of those probability generating functions.

\begin{proposition}\label{prop_defPDS2}
A random variable $X$ is positive discrete stable in the second sense if and only if for all $n\in \N$ there exists a~constant $p_n>0$ such that 
\begin{equation}\label{eq_discretestability3}
\Pcal(\Qcal_{p_n}(z)) = \Pcal^n(z).
\end{equation}
\end{proposition}
\begin{proof}
It follows from the definition \eqref{def_discretestability2} that $X$ is positive discrete stable in the second sense if and only if
$$ \Pcal_{\hat{X}}(z) = \Pcal^n(z).$$
The probability generating function of $\hat{X}$ can be computed in the same way as in Proposition \ref{prop_defPDS}. We obtain
\begin{align*}
\Pcal_{\hat{X}}(z) & = \sum_{k=0}^{\infty} \p(X=k) \left(\E\left[z^{\varepsilon_1(p_n)}\right] \right)^k = \Pcal(\Qcal_{p_n}(z)).
\end{align*}
Hence $X$ is positive discrete stable in the second sense if and only if its probability generating function satisfy the relation 
\begin{equation}
 \Pcal\big(\Qcal_{p_n}(z)\big) = \Pcal^n(z).
\end{equation}
\end{proof}

Further denote by $\q$ a~semigroup generated by the family of probability generating functions $\{\Qcal(z) = \Qcal_{p_n}(z), n \in \N\}$ with operation of superposition. We show that the semigroup $\q$ must be commutative.

\begin{theorem}
Let $X$ be positive discrete stable random variable in the second sense. Then the semigroup $\q$ must be commutative. 
\end{theorem}
\begin{proof}
Let us denote $G(z) = \log \Pcal(z)$. Then \eqref{eq_discretestability3} is equivalent to 
\begin{equation}\label{eq_discretestability4}
n G(z) = G(\Qcal_{p_n}(z)), \quad n \in \N. 
\end{equation}
Let $G(z)$ be a~solution of \eqref{eq_discretestability4}. Then for all $n \in \N$ it must hold $$\Qcal_{p_n}(z) = G^{-1}\left(n G(z)\right).$$ It follows from here that for all $n_1,n_2 \in \N$
\begin{align*}
\Qcal_{p_{n_1}}\left(\Qcal_{p_{n_2}}(z)\right) & = G^{-1}\left(n_1 G\left(G^{-1}\left(n_2G(z)\right)\right)\right) \\
& = G^{-1}\left(n_1 n_2 G(z)\right) \\
& = \Qcal_{p_{n_2}}\left(\Qcal_{p_{n_1}}(z) \right), 
\end{align*}
which means that $\q$ is commutative.
\end{proof}

In the following Subsections we introduce some examples of commutative semigroups $\q$ leading to several possible distributions that are discrete stable in the second sense.

\paragraph{Degenerate portlying operator.}
Assume that the probability generating function $\Qcal(z) = z^n$, i.e.~the portlying distribution is a~degenerate one taking only one value $n$. It is obvious that the semigroup $\q$ is then commutative. This choice of $\Qcal$ leads to a~distribution with probability generating function $\Pcal(z) = z$, i.e. a~degenerate distribution localized at point 1. We are dealing with a~simple summation $n = \sum_{i=1}^n 1.$

\paragraph{Geometric portlying operator.}
Let us consider now geometric distribution with parameter $p \in (0,1)$ with probability generating function $$\Qcal(z) = \frac{pz}{1-(1-p)z}.$$ Such distribution generates a commutative semigroup $\q$, as 
\begin{align*}
\Qcal_{p_1}\left(\Qcal_{p_2}(z)\right) & = \frac{p_1 p_2 z }{1-(1-p_2)z - (1-p_1)p_2 z} = \frac{p_1 p_2 z}{1-z+p_1 p_2 z} \\
& = \Qcal_{p_2}\left(\Qcal_{p_1}(z)\right).
\end{align*}
\begin{proposition}
Let $X$ be an integer-valued random variable with probability generating function 
\begin{equation*}
\Pcal(z) = \exp \left\{-\lambda \left(1-\frac{1}{z}\right)^{\gamma}\right\}. 
\end{equation*}
Then $X$ is positive discrete stable in the second sense.
\end{proposition}
\begin{proof}
Let $\Qcal(z) = \frac{pz}{1-(1-p)z}$ and set $p$ so that $p^{-\gamma} = n$. Then 
\begin{align*}
\log \Pcal(\Qcal(z)) & = -\lambda\left(1-\frac{1-(1-p)z}{pz}\right)^{\gamma} = -\lambda\left(\frac{pz - 1 +(1-p)z}{pz}\right)^{\gamma} \\
& = -\lambda p^{-\gamma} \left(1-\frac{1}{z}\right)^{\gamma} = n \log \Pcal(z).
\end{align*}
Hence by Proposition \ref{prop_defPDS2} the random variable $X$ is positive discrete stable in the second sense.
\end{proof}

It is important to note that the probability generating function $\Pcal(z)$ defines a~non-positive integer-valued random variable.  
%Moreover we can show that there does not exist any non-negative integer-valued random variable positive discrete stable in the second sense with geometric portlying operator.

\paragraph{Portlying operator of Chebyshev type.}
Consider a~probability generating function 
\begin{equation}\label{Q_chebyshev2}
\Qcal(z) = \frac{1}{T_{n}\left(\frac{1}{z}\right) }, \quad n \in \N,
\end{equation} 
where $T_n(x)$ is the Chebyshev polynomial, $T_n(x) = \cos(n\arccos x)$. \cite{klebanov2012} showed that the function $\Qcal(z) = \Qcal_n(z)$ is indeed a~probability generating function of a~random variable with values in $\N$. The semigroup $\q$ generated by the family $\{\Qcal_n(z), n\in \N \}$ is commutative. We have $\Qcal(z) = R^{-1} \circ S \circ R(z),$ where $R(z) = \frac{1}{z}$ and $S(x) = T_n(x)$. Hence 
\begin{align*}
\Qcal_{n_1}\left(\Qcal_{n_2}(z)\right) & =  R^{-1} \circ T_{n_1} \circ R  \circ R^{-1} \circ T_{n_2} \circ R(z) = R^{-1} \circ T_{n_1} \circ T_{n_2} \circ R(z) \\
& = R^{-1} \circ T_{n_2} \circ T_{n_1} \circ R(z) = \Qcal_{n_2}\left(\Qcal_{n_1}(z)\right),
\end{align*}
because Chebyshev polynomials are commutative.

\begin{theorem}\label{PDS2_chebyshev}
Consider the following function 
\begin{equation}\label{pds_Levy}
\Pcal(z) = \left(\frac{1-\sqrt{1-z^2}}{z}\right)^M, \quad M \in \N.
\end{equation}
Then $\Pcal$ is a~probability generating function of a~random variable on $\N$. Moreover if $X$ is an integer-valued random variable with probability generating function $\Pcal$ then $X$ is positive discrete stable in the second sense.
\end{theorem}
\begin{proof}
Let us show first that $\Pcal(z)$ is a~probability generating function. We will consider only the case $M=1$. For $M>1$ the result will follow as $\Pcal(z) = \Pcal_1^M(z)$, where $\Pcal_1(z) = \frac{1-\sqrt{1-z^2}}{z}$, and integer power of a~probability generating function is a~probability generating function of a~sum of i.i.d.~random variables. It is obvious that $\Pcal(1)=1$. We can write $\Pcal(z)$ as
\begin{align*}
\Pcal(z) & =  \frac{1}{z} \left(1-\sqrt{1-z^2}\right) =  \frac{1}{z} - \sum_{k=0}^{\infty}(-1)^k \binom{\frac{1}{2}}{k} z^{2k-1} \\
& = \sum_{k=1}^{\infty} (-1)^{k-1} \binom{\frac{1}{2}}{k}z^{2k-1}.
\end{align*}
The coefficients of the series are all positive, because the binomial coefficient $\binom{\frac{1}{2}}{k}$ involves $(k-1)$ negative factors.

Now let us show that $X$ is positive discrete stable in the second sense. Let $\Qcal(z)$ be as in \eqref{Q_chebyshev2}. Then 
\begin{align*}
\Pcal(\Qcal(z)) & =  T_n \left(\frac{1}{z} \right) - \sqrt{T_n^2 \left(\frac{1}{z} \right)-1}.
\end{align*}
We can use the explicit expression of Chebyshev polynomial to obtain
\begin{align*}
T_n \left(\frac{1}{z} \right)  & = \frac{(1+\sqrt{1-z^2})^n+(1-\sqrt{1-z^2})^n}{2 z^ n}\\
\intertext{and}
\sqrt{T_n^2 \left(\frac{1}{z} \right)-1} & = \frac{(1+\sqrt{1-z^2})^n-(1-\sqrt{1-z^2})^n}{2 z^ n}.
\end{align*}
From here we see that $$\Pcal(\Qcal(z)) = \Pcal(z)^n.$$ Hence by Proposition \ref{prop_defPDS2} the random variable $X$ is positive discrete stable in the second sense.
\end{proof}

In the proof of the theorem we showed that $$\Pcal(z) =\sum_{k=1}^{\infty} (-1)^{k-1} \binom{\frac{1}{2}}{k}z^{2k-1},$$ so the probabilities $\p(X=k)$ are given as $(-1)^{k-1} \binom{\frac{1}{2}}{k}$ for all odd $k > 0$ and 0 otherwise. 

\begin{remark} 
The probability distribution with generating function \eqref{pds_Levy} for $M=1$ is known (see \cite[\S XI.3]{feller1}) as a distribution of the first passage time of a random walk through +1. Let us consider a sequence of Bernoulli trials $X_1, X_2, \dots$ with probability $p=1/2$, i.e. $\p(X_i = 1) = 1-\p(X_i=-1) = 1/2$ and denote $S_n = X_1+X_2 + \dots + X_n$, $S_0=0$. Then the random walk $S_n$ passes through $+1$ for the first time at time $m$ if $$S_1 \leq 0,  \dots S_{m-1} \leq 0, \quad S_m = 1.$$ The probability of this event is given by the probability generating function \eqref{pds_Levy}. 

In continuous case we have a similar result. The first passage time of a Brownian motion through a level $a>0$ has L\'evy distribution, a special case of stable distribution with $\alpha = 1/2$. 
\end{remark}

The discrete stable distribution with probability generating function \eqref{pds_Levy} with $M=1$ can be considered a discrete analogy of L\'evy distribution as is shown in the following Theorem.

\begin{theorem}
Discrete stable random distribution in the second sense with probability generating function $$\Pcal(z) = \frac{1-\sqrt{1-z^2}}{z}$$ belongs to the domain of normal attraction of stable distribution $\s\left(\tfrac{1}{2},1,1,0\right)$, i.e.~L\'evy distribution.
\end{theorem}

\begin{proof}
Let $X_1, X_2, \dots, X_n$ be i.i.d.~positive discrete stable random variables in the second sense with probability generating function $\Pcal(z)$. The characteristic function of $X_1$ is equal to $f(t) = \Pcal\left(e^{\im t}\right)$. Denote $$S_n = \frac{1}{n^2} \sum_{i=1}^n X_i.$$ Then the characteristic function of $S_n$ is equal to $$f_n(t) = f^n\left(t/n^2\right) \longrightarrow \exp\left\{-\sqrt{2}(-\im t)^{1/2}\right\}, \quad \text{as} \quad n \to \infty.$$ Moreover $(-\im t)^{1/2} = \frac{1}{\sqrt{2}}|t|^{1/2} \left(1-\im \, \mathrm{sgn}(t)\right).$
\end{proof}

Consider now a~slightly different setting with portlying operator with probability generating function 
\begin{equation}\label{Q2_chebyshev}
\Qcal(z) = \left(\frac{1}{T_{n}\left(\frac{1}{z^m}\right) }\right)^{1/m}, n, m \in \N.
\end{equation} 
As was noted in \cite{klebanov2012}, $\Qcal$ is a~probability generating function of a~random variable with values in $m\N$.

\begin{theorem}
Let $X$ be an integer-valued random variable with probability generating function 
\begin{equation*}
\Pcal(z) = \frac{1-\sqrt{1-z^{2m}}}{z^m}, \quad m \in \N.
\end{equation*}
Then $X$ is positive discrete stable in the second sense.
\end{theorem}

\begin{proof}
We have
\begin{align*}
\Pcal(\Qcal(z)) & = T_n\left(\frac{1}{z^m}\right) - \sqrt{T_n\left(\frac{1}{z^m}\right)^2 - 1},
\intertext{and using results from the proof of Theorem \ref{PDS2_chebyshev},}
\Pcal(\Qcal(z)) & = \frac{1-\sqrt{1-z^{2m}}}{z^m}.
\end{align*}
\end{proof}

%%%%%%%%%%%%%%%%%%%%%%%%%%%%%%%%%%%%%%%%%%%%%%%%%%%%%%%%%%%%%%%%%%%%%%%%%%%%%%%%%%%%%%%%%%%%%%%%%%%%%%%%%%%%%%%%%%%%%%%%%%%%%%%%%%%%%%%%%%%%%%%%%%%%%%%%%%%%%%%%%%%%%%%%%%%%%%%%%%%%%%%%%%
%%%%%%%%%%%%%%%%%%%%%%%%%%%%%%%%%%%%%%%%%%%%%%%%%%%%%%%%%%%%%%%%%%%%%%%%%%%%%%%%%%%%%%%%%%%%%%%%%%%%%%%%%%%%%%%%%%%%%%%%%%%%%%%%%%%%%%%%%%%%%%%%%%%%%%%%%%%%%%%%%%%%%%%%%%%%%%%%%%%%%%%%%%
%%%%%%%%%%%%%%%%%%%%%%%%%%%%%%%%%%%%%%%%%%%%%%%%%%%%%%%%%%%%%%%%%%%%%%%%%%%%%%%%%%%%%%%%%%%%%%%%%%%%%%%%%%%%%%%%%%%%%%%%%%%%%%%%%%%%%%%%%%%%%%%%%%%%%%%%%%%%%%%%%%%%%%%%%%%%%%%%%%%%%%%%%%
%%%%%%%%%%%%%%%%%%%%%%%%%%%%%%%%%%%%%%%%%%%%%%%%%%%%%%%%%%%%%%%%%%%%%%%%%%%%%%%%%%%%%%%%%%%%%%%%%%%%%%%%%%%%%%%%%%%%%%%%%%%%%%%%%%%%%%%%%%%%%%%%%%%%%%%%%%%%%%%%%%%%%%%%%%%%%%%%%%%%%%%%%%

\subsection{On third definition of discrete stable distributions}

In this Section we give a~definition of discrete stability that generalizes the third definition of strict stability \eqref{eq_def_stability3} for discrete random variables. As it turns out, this definition is a~combination of the two previous definitions. 

\begin{definition}\label{def_PDS3}
Let $X, X_1$ and  $X_2$ be independent and identically distributed non-negative integer-valued random variables. Assume that for any positive numbers $p_1$ and $p_2$ there exists a~positive number $p$ such that 
\begin{equation}\label{def_discretestability3}
\tilde{X}(p) \stackrel{d}{=} \tilde{X}_1(p_1) + \tilde{X}_2(p_2), \quad \text{where} \quad 
\tilde{X}(p) = \sum_{j=1}^{X} \varepsilon_j(p) 
\end{equation}
and $\varepsilon_j(p)$ are i.i.d.~non-negative integer-valued random variables. Then we say that $X$ is \textit{positive discrete stable random variable in the third sense}.
\end{definition}

Let us denote the probability generating functions of the random variables $X$ and $\varepsilon(p)$ by $\Pcal(z) = \E[z^X]$ and $\Qcal_{p}(z) = \E[z^{\varepsilon(p)}]$ respectively. Let us again denote the semigroup generated by $\{\Qcal_p, p \in \Delta\}$ with operation of superposition by $\q$. There is an equivalent definition of positive discrete stability in the third sense in terms of those probability generating functions, following directly from the Definition.

\begin{proposition}\label{prop_defPDS3}
A random variable $X$ is positive discrete stable in the third sense if and only if for any positive numbers $p_1$ and $p_2$ there exists a~positive number $p$ such that 
\begin{equation}\label{eq_discretestability5}
\Pcal(\Qcal_{p}(z)) = \Pcal(\Qcal_{p_1}(z))\Pcal(\Qcal_{p_2}(z)).
\end{equation}
\end{proposition}

We can show that every random variable positive discrete stable in the first sense is also positive discrete stable in the third sense.

\begin{theorem}
Let $X$ be positive discrete stable in the first sense. Then $X$ is positive discrete stable in the third sense.  Moreover \eqref{def_discretestability3} holds with $$p^\gamma = p_1^{\gamma} + p_2^{\gamma}.$$
\end{theorem}
\begin{proof}
Let $X$ be positive discrete stable in the first sense, and let $X_1, X_2, \dots$ be independent copies of $X$. Then the semigroup $\q$ is commutative, $p \in \Delta = (0,1)$ and for any $n \geq 2$ there exists a~constant $p_n \in (0,1)$ such that $$X \stackrel{d}{=} \sum_{i=1}^n p_n \odot X_i.$$ From Theorem \ref{th:indexofstability} we know that $p_n = n^{-1/\gamma}.$ Let $p_1, p_2 \in \Delta$. Then for all $n_1, n_2 \geq 2$ 
\begin{align*} 
p_1 \odot X_1 + p_2 \odot X_2 & \stackrel{d}{=} \sum_{i=1}^{n_1} p_1 p_{n_1} \odot X_i + \sum_{j=n_1+1}^{n_1 + n_2} p_1 p_{n_2} \odot X_{j}. 
\end{align*}
If $p_1^{\gamma}, p_2^{\gamma}$ are rational, then we can find $n_1, n_2, p$ such that
\begin{align*}
p_1 p_{n_1} & = p p_{n_1 + n_2}, \\
p_2 p_{n_2} & = p p_{n_1 + n_2}, 
\intertext{or equivalently}
p_1^{\gamma} & = p^{\gamma} \frac{n_1}{n_1 +n_2}, \\
p_2^{\gamma} & = p^{\gamma} \frac{n_2}{n_1 + n_2}.
\end{align*}
But then, with $n=n_1 + n_2$
$$p_1 \odot X_1 + p_2 \odot X_2  = \sum_{i=1}^n p p_n \odot X_i = p \odot X.$$ Moreover $p_1,p_2,p$ satisfy the relationship $p_1^{\gamma} + p_2^{\gamma} = p^{\gamma}.$ By continuity argument it follows that \eqref{def_discretestability3} hold for any choice of $p_1, p_2$ with $p$ such that $p_1^{\gamma} + p_2^{\gamma} = p^{\gamma}.$
\end{proof}

Under some additional conditions we may show that the opposite statement holds true as well.
\begin{theorem}
Let $X$ be positive discrete stable in the third sense and assume that the semigroup $\q$ is commutative, $\Delta = (0,1)$ and that there exists a~constant $\gamma > 0$ such that $$p^{\gamma} = p_1^{\gamma} + p_2^{\gamma}.$$ Then $X$ is positive discrete stable in the first sense. 
\end{theorem}
\begin{proof}
We may show this by induction. Because $X$ is positive discrete stable in the third sense, we have for $p_1 = p_2 = 2^{-1/\gamma}$ that $$X \stackrel{d}{=} \tilde{X}_1(p_2) + \tilde{X}_2(p_2).$$ Let $n \geq 2$ and let us assume that $$X \stackrel{d}{=} \sum_{i=1}^n \tilde{X}_i(p_n), \quad \text{with} \quad p_n = n^{-1/\gamma}.$$ Denote $Y = \sum_{i=1}^n \tilde{X}_i(p_n)$ and let $p = \left(\frac{n}{n+1}\right)^{1/\gamma}$. Because $X$ is positive discrete stable in the third sense, $Y \stackrel{d}{=} X$ and $p^{\gamma} + p_{n+1}^{\gamma} = 1$, we have $$X \stackrel{d}{=} \tilde{Y}(p) + \tilde{X}_{n+1}(p_{n+1}).$$ The probability generating function of the right-hand side is
\begin{align*}
\Pcal_Y(\Qcal_p(z)) \Pcal(\Qcal_{p_{n+1}}(z)) &  = \Pcal^n(\Qcal_{p_n}(\Qcal_p(z))) \Pcal(\Qcal_{p_{n+1}}(z)) \\
& = \Pcal^n(\Qcal_{p_n p}(z)) \Pcal(\Qcal_{p_{n+1}}(z)) \\
 & = \Pcal^{n+1}(\Qcal_{p_{n+1}}(z)),
\end{align*}
 because $\q$ is commutative and $p p_n = p_{n+1}$. Therefore $$X \stackrel{d}{=} \sum_{i=1}^{n+1} \tilde{X}_i(p_{n+1}).$$
\end{proof}

%
%Positive discrete stable random variables in the second sense are also positive discrete stable in the third sense.
%
%\begin{theorem}
%Let $X$ be positive discrete stable in the second sense. Then $X$ is positive discrete stable in the third sense.
%\end{theorem}
%\begin{proof}
%Let $X$ be positive discrete stable in the second sense. That means that for $n = 2$ there exists a~constant $p_n \in (0,1)$ such that $$\Pcal(\Qcal_{p_n}(z)) = \Pcal(z)^2.$$ Therefore $X$ is positive discrete stable in the third sense with $p = p_n$ and $p_1 = p_2 = 1$.\\
%\end{proof}

\paragraph{Binomial thinning operator.}
Let us consider the case of the binomial thinning operator with probability generating function $\Qcal(z) = (1-p) + pz$. Then a~random variable $X$ with probability generating function $\Pcal(z) = \exp\left\{-\lambda(1-z)^{\gamma}\right\}$ is positive discrete stable in the third sense, as \eqref{eq_discretestability5} holds if $$p^{\gamma} = p_1^{\gamma} + p_2^{\gamma}.$$

\paragraph{Modified geometric thinning operator.}
We can verify that the positive discrete stable random variable in the first sense with modified geometric thinning operator is also positive discrete stable in the third sense. Let $X$ be a~positive discrete random variable in the first sense with probability generating function $\Pcal(z) = \exp\left\{-\lambda \left(\frac{1-z}{1-\kappa z}\right)^{\gamma}\right\}$. Then $$\Pcal\left(\Qcal_{p}(z)\right) = \exp\left\{-\lambda p^{\gamma} \left(\frac{1-z}{1-\kappa z}\right)^{\gamma}\right\}.$$ Thus \eqref{eq_discretestability5} holds if $$p^{\gamma} = p_1^{\gamma} + p_2^{\gamma}.$$

\paragraph{Chebyshev thinning operator.}
In the same manner we see that a~positive discrete stable random variable  in the first sense with Chebyshev thinning operator  $X$ is positive discrete stable in the third sense. Let $\Pcal$ be as in \eqref{pz_modifiedpositivediscretestable} and $\Qcal$ as in \eqref{Q_chebyshev}. We have $$\Pcal(\Qcal_p(z)) = \left[\Pcal(z)\right]^{p^{\gamma}}.$$ Therefore again \eqref{eq_discretestability5} holds if $$p^{\gamma} = p_1^{\gamma} + p_2^{\gamma}.$$

\paragraph{Chebyshev portlying operator.}
Now let's look at an example with Chebyshev portlying operator with probability generating function $\Qcal_{n}(z) = 1/T_n(1/z)$. Then a~random variable $X$ with probability generating function $\Pcal(z) = \left(1-\sqrt{1-z^2}\right)/z$ is positive discrete stable in the third sense, as \eqref{eq_discretestability5} holds if $$n = n_1 + n_2.$$

%%%%%%%%%%%%%%%%%%%%%%%%%%%%%%%%%%%%%%%%%%%%%%%%%%%%%%%%%%%%%%%%%%%%%%%%%%%%%%%%%%%%%%%%%%%%%%%%%%%%%%%%%%%%%%%%%%%%%%%%%%%%%%%%%%%%%%%%%%%%%%%%%%%%%%%%%%%%%%%%%%%%%%%%%%%%%%%%%%%%%%%%%%
%%%%%%%%%%%%%%%%%%%%%%%%%%%%%%%%%%%%%%%%%%%%%%%%%%%%%%%%%%%%%%%%%%%%%%%%%%%%%%%%%%%%%%%%%%%%%%%%%%%%%%%%%%%%%%%%%%%%%%%%%%%%%%%%%%%%%%%%%%%%%%%%%%%%%%%%%%%%%%%%%%%%%%%%%%%%%%%%%%%%%%%%%%
%%%%%%%%%%%%%%%%%%%%%%%%%%%%%%%%%%%%%%%%%%%%%%%%%%%%%%%%%%%%%%%%%%%%%%%%%%%%%%%%%%%%%%%%%%%%%%%%%%%%%%%%%%%%%%%%%%%%%%%%%%%%%%%%%%%%%%%%%%%%%%%%%%%%%%%%%%%%%%%%%%%%%%%%%%%%%%%%%%%%%%%%%%
%%%%%%%%%%%%%%%%%%%%%%%%%%%%%%%%%%%%%%%%%%%%%%%%%%%%%%%%%%%%%%%%%%%%%%%%%%%%%%%%%%%%%%%%%%%%%%%%%%%%%%%%%%%%%%%%%%%%%%%%%%%%%%%%%%%%%%%%%%%%%%%%%%%%%%%%%%%%%%%%%%%%%%%%%%%%%%%%%%%%%%%%%%

\section{Properties of positive discrete stable random variables}\label{sec:pdsG}
Distributions, that are discrete stable in the first sense, form the widest and most interesting class of distributions, and in the following Sections we study properties of the distributions with thinning operator of geometric type.

To remind the definition, a~non-negative integer-valued random variable $X$ is said to be positive discrete stable in the first sense, if 
\begin{equation}\label{eq_defpds2}
X \stackrel{d}{=} \sum_{j=1}^{n} \tilde{X}_j, \quad \text{where} \quad \tilde{X}_j = \sum_{i=1}^{X_j} \varepsilon_i^{(j)},
\end{equation}
where $X_1, X_2, \dots$ are independent copies of $X$ and $\varepsilon_i^{(j)}$ are i.i.d.~non-negative integer-valued random variables. Throughout this Section we will assume that the random variables  $\varepsilon_i^{(j)}$ come from modified geometric distribution $\mathcal{G}(p,\kappa,m)$ with probability generating function $\Qcal$ of the form

\begin{equation}\label{Qz2}
\Qcal(z) = \left( \frac{(1-p)+(p-\kappa)z^m}{(1-p\kappa)-\kappa(1-p)z^m} \right)^{\frac{1}{m}},\quad 
\begin{array}{l}
\{0 \leq \kappa < 1, \; \, 0< p < 1, \; \, m =1\} \\
\text{ or }  \\
\{0 < p < \kappa < 1, \; \,  \; \, m \in \N, m > 1\}.
\end{array}
\end{equation}

We remind that $\Qcal(z)$ can be decomposed as $\Qcal(z) = S^{-1} \circ B_p \circ S(z),$ where $B_p(z)= pz + (1-p)$ and $$S(z) = \frac{(1-\kappa)z^m}{1-\kappa z^m}, \quad S^{-1}(y) = \left(\frac{y}{(1-\kappa)+\kappa y}\right)^{\frac{1}{m}}.$$

\begin{theorem}
A non-negative integer-valued random variable $X$ is positive discrete stable with $\mathcal{G}$ thinning operator if and only if $\Qcal$ takes form \eqref{Qz2} and the probability generating function $\Pcal(z) = \E z^X$ is given as
\begin{equation}\label{pgf_pds}
\Pcal(z) = \exp\left\{-\lambda\left(\frac{1- z^m}{1-\kappa z^m}\right)^{\gamma}\right\} \quad \text{with} \quad \gamma \in (0,1], \; \lambda > 0,\; \kappa \in [0,1), \; m \in \N. 
\end{equation}
\end{theorem}

\begin{proof}
Let $h(z) = \log \Pcal(z)$. From Proposition \ref{prop_defPDS} it follows that $X$ is positive discrete stable if and only if $h(z) = n h(\Qcal(z))$ for all $n$. Set $$h(z) = -\lambda\left(\frac{1- z^m}{1-\kappa z^m}\right)^{\gamma}$$ and select $\gamma$ such that $1/p^\gamma = n.$  We see that $$\frac{1-z^m}{1-\kappa z^m} = 1 - \frac{(1-\kappa)z^m}{1-\kappa z^m} = 1 - S(z).$$ Therefore, using the decomposition of $\Qcal(z)$,
\begin{align*}
n h(\Qcal(z)) & = - \lambda n \left(1 - S(\Qcal(z))\right)^{\gamma} = -\lambda n \left(1 - B_p(S(z))\right)^{\gamma} \\
& =  - \lambda n \left(p - p S(z) \right)^{\gamma} = -\lambda n p^{\gamma}(1-S(z))^{\gamma} \\
& = -\lambda \left(\frac{1-z^m}{1-\kappa z^m}\right)^{\gamma} = h(z).
\end{align*}
\end{proof}

The parameter $m$ determines the size of the lattice of the distribution. We will denote positive discrete stable random variable (and associated distribution) by $\pds^m(\gamma,\lambda,\kappa)$. In the case when $m$ is omitted we will understand that $m=1$. If moreover $\kappa$ is omitted, we will understand that $\kappa = 0$, in which case the discrete stable distribution reduces to the discrete stable distribution as it was introduced in \cite{steutel}. %In Figure \ref{fig:pds} the probabilities of $\pds(\gamma,\lambda,\kappa)$ random variables are shown for different values of parameters. The probabilities were obtained using the classical inverse Fourier transform theorem (see, for example, \cite{lachout}) and the fast Fourier transform algorithm.

The characteristic function is given as 
$$f(t) = \exp\left\{-\lambda\left(\frac{1 - e^{\im t m}}{1-\kappa e^{\im t m}}\right)^{\gamma}\right\}.$$
The case of $\gamma = 1$ is a~special one as it leads to a~distribution with finite variance and exponential tails.  As a~simple corollary we obtain Poisson distribution by taking $\kappa = 0$ and $\gamma = 1$.

% #################################################### CHARACTERIZATIONS ####################################################

\subsection{Characterizations}

In this Subsection we present several characterizations of positive discrete stable random variables. 

\begin{theorem}
Let $\gamma \in (0,1)$ be a~given parameter. Let $X, X_1, X_2, \dots$ be i.i.d.~non-negative integer-valued random variables and $Y$ be a~non-negative integer-valued random variable, independent of the sequence $X_1, X_2,\dots$. Then $X$ is positive discrete stable $\pds(\gamma,\lambda)$ random variable if and only if
\begin{equation}\label{char_pds}
X \stackrel{d}{=} \sum_{j=1}^Y Y^{-1/\gamma} \odot X_j,  \quad \text{where} \quad p \odot X = \sum_{i=1}^X \varepsilon_i(p)
\end{equation}
and $\varepsilon_i(p)$ are i.i.d.~Bernoulli random variables with probability generating function $\Qcal_p(z) = 1-p+pz.$
\end{theorem}

\begin{proof}
First let us show that if $X$ is $\pds(\gamma,\lambda)$ then it has the representation \eqref{char_pds}. Let $\Pcal(z)$ be the probability generating function of $X$. The probability generating function of the right-hand side of \eqref{char_pds} can be computed in the following way. 
\begin{align*}
\E\left[z^{\sum_{j=1}^Y Y^{-1/\gamma} \odot X_j}\right] & = \E\left[\E\left[z^{\sum_{j=1}^Y Y^{-1/\gamma} \odot X_j}|Y\right] \right] = \E\left[\Pcal_X^Y\left(\Qcal_{Y^{-1/{\gamma}}}(z)\right) \right] \\
& = \E\left[\exp\left\{-\lambda Y \left(1-\Qcal_{Y^{-1/{\gamma}}}(z)\right)^{\gamma}\right\}\right] = \E\left[\exp\left\{-\lambda Y Y^{-1} \left(1-z\right)^{\gamma}\right\}\right] \\
& = \exp\left\{-\lambda\left(1-z\right)^{\gamma}\right\} = \Pcal(z).
\end{align*}

The proof of the inverse statement is more complicated and relies on the method if intensively monotone operators. The condition \eqref{char_pds} can be translated into the form of probability generating functions as
\begin{equation}\label{char_pds2}
\Pcal(z) = \sum_{k=0}^{\infty} \p(Y=k) \prod_{j=1}^k \Pcal\left(\Qcal_{k^{-1/\gamma}}(z)\right).
\end{equation} 
Put $G(z) = \log \Pcal(z)$ and $h(z) = G(z)/(1-z)^{\gamma}$. Then we can rewrite \eqref{char_pds2} as
\begin{align}
h(z) & = (1-z)^{-\gamma} \sum_{k=0}^{\infty} \p(Y=k) \sum_{j=1}^k \left(1-\Qcal_{k^{-1/\gamma}}(z)\right)^{\gamma} h\left(\Qcal_{k^{-1/\gamma}}(z)\right) \label{char_pds3}\\
& = (1-z)^{-\gamma} \sum_{k=0}^{\infty} \p(Y=k) \left(1-z\right)^{\gamma} h\left(\Qcal_{k^{-1/\gamma}}(z)\right) \notag \\
& = \sum_{k=0}^{\infty} \p(Y=k) h\left(\Qcal_{k^{-1/\gamma}}(z)\right). \notag
\end{align}
Let $A$ be an operator acting on $g \in C[0,1]$ such that 
\begin{align*}
(Ag)(z) & = \left\{ \begin{array}{ll}
\sum_{k=0}^{\infty}\p(Y=k) g\left(\Qcal_{k^{-1/\gamma}}(z)\right), & z < 1\\
 g(0), & z=1.
\end{array} \right. 
\end{align*}
We can verify that $A$ is an intensively monotone operator (see \cite{kakosyan}) and that $Ag \in C[0,1]$. It is clear that $A a~= a$ for all constant functions $a$. It follows from \cite[Theorem 1.1.2]{kakosyan} that the only solution of \eqref{char_pds3} is identically equal to a constant. Hence $h(z) = -\lambda$ and $$\Pcal(z) = \exp\left\{-\lambda(1-z)^{\gamma}\right\}.$$ 
\end{proof}

\begin{theorem}\label{char_theorem_pds}
Let $\gamma, \gamma' \in (0,1]$ and assume that $\gamma' \leq \gamma$. Let $\s_\gamma$ be a~$\gamma$-stable random variable with Laplace transform $\exp\{-u^{\gamma}\}$. Then 
$$\pds(\gamma',\lambda,\kappa) \stackrel{d}{=} \pds\left(\gamma'/\gamma, \lambda^{1/\gamma} \s_{\gamma},\kappa\right).$$
\end{theorem}

\begin{proof}
The characteristic function of the right-hand side can be computed as 
\begin{align*}
\E\left[\exp\left\{\im t \pds\left(\gamma'/\gamma, \lambda^{1/\gamma} \s_{\gamma}, \kappa\right)\right\}\right] & = \E\left[\exp\left\{-\lambda^{1/\gamma} \s_{\gamma} \left(\frac{1- e^{\im t}}{1-\kappa e^{\im t}}\right)^{\gamma'/\gamma} \right\}\right] \\
& = \exp\left\{-\lambda \left(\frac{1- e^{\im t}}{1-\kappa e^{\im t}}\right)^{\gamma'} \right\} \\
& = \E\big[\exp\left\{\im t \pds\left(\gamma', \lambda, \kappa\right)\right\}\big].
\end{align*}
\end{proof}

The following Corollary can be applied for simulations of positive discrete stable random variables. 

\begin{corollary}\label{pds_simul}
Let $Y, Y_1, Y_2, \dots$ be a~sequence of i.i.d.~random variables with geometric distribution, $\p(Y=n) = (1-\kappa) \kappa^{n-1},n \geq 1.$ Let $N$ be a~random variable, independent of the sequence $Y_1,Y_2, \dots$, with Poisson distribution with random intensity $\lambda^{-1/\gamma} \s_{\gamma}$, where $\s_{\gamma}$ is a~$\gamma$-stable random variable with Laplace transform $\exp\{-u^{\gamma}\}$. Then 
$$\sum_{j=1}^N Y_j$$ has the same distribution as a~positive discrete stable random variable $\pds(\gamma,\lambda,\kappa)$.
\end{corollary}
\begin{proof}
Let $X = \sum_{j=1}^N Y_j$. Then $X$ is a~compound Poisson random variable with random intensity $\lambda^{1/\gamma} \s_{\gamma}$ and jumps $Y_1, Y_2, \dots$ with characteristic function $$g(t) = \frac{(1-\kappa)e^{\im t}}{1-\kappa e^{\im t}}.$$ The characteristic function of a~compound Poisson random variable with intensity $\tau$ and characteristic function of jumps $h(t)$ is $\exp\{-\tau(1-h(t))\}.$ Therefore $X$ is in fact $\pds(1,\lambda^{1/\gamma}\s_{\gamma},\kappa)$.  We thus obtain the result from the previous Theorem \ref{char_theorem_pds} with $\gamma'= \gamma$.
\end{proof}

% #################################################### MOMENTS ####################################################

\subsection{Moments}

\begin{theorem}\label{pds_moments}
Let $X$ be $\pds(\gamma,\lambda,\kappa)$ random variable with $\gamma = 1$ and $\kappa>0$. Then the $n$-th factorial moment can be computed using the following formula
\begin{equation}
\E\left[(X)_n\right] = \frac{\kappa^n}{ (1-\kappa)^n} n! \sum_{s=0}^{n-1} \frac{1}{(s+1)!} \binom{n-1}{s} \frac{\lambda^{s+1}}{\kappa^{s+1}}.
\end{equation} 
\end{theorem}
\begin{proof}
Let $\Pcal(z)$ be the probability generating function of $X$. The $n$-th factorial moment of discrete random variable can be computed as the value of the $n$-th derivative of the probability generating function at point 1, i.e.
$$\E\left[(X)_n\right] = \left.\frac{\dd^n}{\dd z^n} \Pcal(z)\right|_{z=1}.$$
Since $\Pcal(z)=\exp\{g(z)\},$ with $$g(z) = -\lambda \left(1-(1-\kappa)\frac{z}{1-\kappa z}\right),$$ we compute the $n$-th derivative using the Bruno's formula (\cite{bruno}) 
\begin{align*}
\left.\frac{\dd^n}{\dd z^n} \Pcal(z)\right|_{z=1} = \sum_{k=1}^n \Pcal(1) B_{n,k}(g'(1), g''(1), \dots, g^{(n-k+1)}(1)),
\end{align*}
where $B_{n,k}(x_1, \dots, x_{n-k+1})$ is the Bell's polynomial,
\begin{equation}\label{bell}
B_{n,k}(x_1, \dots, x_{n-k+1}) = \sum_{i_1, \dots, i_{n-k+1}} \frac{n!}{i_1! i_2! \dots i_{n-k+1}!}\left(\frac{x_1}{1!}\right)^{i_1}\left(\frac{x_2}{2!}\right)^{i_1}\cdots \left(\frac{x_{n-k+1}}{(n-k+1)!}\right)^{i_{n-k+1}},
\end{equation} 
where we sum over all possible combinations such that $i_1+2 i_2\dots+(n-k+1) i_{n-k+1} = n$ and $ i_1+ i_2\dots+ i_{n-k+1} = k $.
By differentiating the function $g(z)$ we obtain 
$$g^{(i)}(1)= i! \lambda \frac{\kappa^{i-1}}{(1-\kappa)^i}.$$ Plugging that into the Bell's polynomial we obtain 
\begin{align*}
B_{n,k}\left(g'(1), g''(1), \dots, g^{(n-k+1)}(1)\right) & = \sum_{i_1, \dots, i_{n-k+1}} \frac{n!}{i_1! i_2! \dots i_{n-k+1}!}\prod_{j=1}^{n-k+1}\left(\frac{g^{(j)}(1)}{j!}\right)^{i_j} \\
& = \sum_{i_1, \dots, i_{n-k+1}} \frac{n!}{i_1! i_2! \dots i_{n-k+1}!} \prod_{j=1}^{n-k+1}\left(\frac{\lambda \kappa^{j-1}}{(1-\kappa)^{j}}\right)^{i_j} \\
& =  \sum_{i_1, \dots, i_{n-k+1}} \frac{n!}{i_1! i_2! \dots i_{n-k+1}!} \frac{\lambda^k \kappa^n}{\kappa^k (1-\kappa)^n } \\
& = \frac{\lambda^k \kappa^n}{\kappa^k (1-\kappa)^n} B_{n,k}(1!,2!, \dots, (n-k+1)!)\\
& = \frac{\lambda^k \kappa^n}{\kappa^k (1-\kappa)^n} \binom{n}{k} \binom{n-1}{k-1} (n-k)!.
\end{align*} 
Hence the $n$-th factorial moment is 
\begin{align*}
\E\left[(X)_n\right] & = \sum_{k=1}^n \frac{\lambda^k \kappa^n}{\kappa^k (1-\kappa)^n} \binom{n}{k} \binom{n-1}{k-1} (n-k)! \\
& = \frac{\kappa^n}{(1-\kappa)^n} \sum_{k=1}^n \frac{\lambda^k }{\kappa^k } \frac{n!}{k!}\binom{n-1}{k-1}.
\end{align*}
The result follows from here by setting $s = k-1$.
\end{proof}

% ################################################### PROBABILITIES ####################################################

\subsection{Probabilities}

In the next Theorem we show connection between the probabilities of a positive discrete stable random variable and moments of a tempered stable random variable. 

\begin{theorem}\label{pds_ts}
Let $X$ be a $\pds(\gamma,\lambda)$ random variable with $\gamma<1$. Let $Y$ be a tempered stable random variable with characteristic function $f_Y(t) = \exp\{-(\lambda^{1/\gamma}-\im t)^{\gamma} + \lambda\}.$ Then we can write the probabilities $\p(X = k)$ as 
$$\p(X=k) = e^{-\lambda} \frac{\lambda^{k/\gamma}}{k!} \E Y^k.$$
\end{theorem}
Before we proceed to the proof of the Theorem, we state a~simple Lemma.

\begin{lemma}\label{lem:ts}
Let $S_{\gamma}$ be $\gamma$-stable random variable with Laplace transform $L(u)= \E e^{-u S_{\gamma}} = \exp\{-u^{\gamma}\}$ and density function $p(x)$. Let $\theta > 0$. Let $Y$ be a~random variable with density function $$p_Y(x) = e^{-\theta x} p(x)/L(\theta).$$ Then $Y$ is a~tempered stable random variable with characteristic function $$f(t) = \exp\{-(\theta-\im t)^{\gamma} + \theta^{\gamma}\}.$$
\end{lemma}
\begin{proof}
We may compute the characteristic function of $Y$ as follows:
\begin{align*}
f_Y(t) & = \E e^{\im t Y} = \int_0^{\infty} e^{\im t x} p_Y(x) \dd x = \int_0^{\infty} e^{\im t x} e^{-\theta x} p(x)/ L(\theta) \dd x \\
& = e^{\theta^{\gamma}} \int_0^{\infty} \exp\{-(\theta - \im t)x\}\, p(x) \dd x  \\
& = e^{\theta^{\gamma}} L(\theta-\im t) = \exp\left\{-(\theta - \im t)^{\gamma} + \theta^{\gamma}\right\}.
\end{align*}
\end{proof}
Now we can prove the Theorem.
\begin{proof}[Proof of Theorem \ref{pds_ts}]
It follows from Theorem \ref{char_theorem_pds} that a~positive discrete stable random variable $\pds(\gamma,\lambda)$ is a~Poisson random variable with random intensity $\lambda^{1/\gamma}S_{\gamma}$, where $S_{\gamma}$ is a~$\gamma$-stable random variable with Laplace transform $L(u) = \exp\{-u^{\gamma}\}$ and density function $p(x)$. Therefore the probabilities $\p(X = k)$ can be computed as
\begin{align*}
\p(X = k) & = \int_0^{\infty} e^{-\lambda^{1/\gamma} s}\frac{(\lambda^{1/\gamma} s)^k}{k!} p(s) \dd s \\
& = \frac{\lambda^{k/\gamma}}{k!} L(\lambda^{1/\gamma}) \int_0^{\infty} s^k e^{-\lambda^{1/\gamma} s} p(s) / L(\lambda^{1/\gamma}) \dd s.  
\end{align*}
But $e^{-\lambda^{1/\gamma} s} p(s) / L(\lambda^{1/\gamma})$ is a~density function of a~tempered stable random variable $Y$ with characteristic function $f(t) = \exp\{-(\lambda^{1/\gamma}-\im t)^{\gamma} + \lambda\}.$ Therefore 
\begin{align*}
\p(X = k) & =  \frac{\lambda^{k/\gamma}}{k!} L(\lambda^{1/\gamma}) \int_0^{\infty} s^k p_Y(s) \dd s \\
& = \frac{\lambda^{k/\gamma}}{k!} L(\lambda^{1/\gamma}) \E Y^k. 
\end{align*}

\end{proof}

\begin{theorem}
Let $X$ be a $\pds(\gamma,\lambda,\kappa)$ random variable with $\gamma = 1$ and $\kappa>0$. Then the probability $\p(X=m)$ for $m\geq 1$  can be computed using the following formula
\begin{equation}\label{pm_pds}
\p\left(X = m \right) = e^{-\lambda} \sum_{s=0}^{m-1} \frac{\lambda^{s+1}}{(s+1)!}\binom{m-1}{s} \kappa^{m-s-1} (1-\kappa)^{s+1}.
\end{equation} 
\end{theorem}
\begin{proof}
We compute the probabilities by expanding the probability generating function into power series. 
\begin{align*}
\mathcal{P}(z)& = \exp\left\{-\lambda\left(1-(1-\kappa)\frac{z}{1-\kappa z }\right)\right\} \\
& = e^{-\lambda}+ e^{-\lambda} \sum_{n=1}^{\infty} \frac{\lambda^n}{n!} (1-\kappa)^n \frac{z^n}{(1-\kappa z)^n} \\
&  =e^{-\lambda}+e^{-\lambda} \sum_{n=1}^{\infty} \sum_{j=0}^{\infty} \frac{\lambda^n}{n!} (1-\kappa)^n \kappa^j \binom{n+j-1}{j} z^{n+j} \\
& = e^{-\lambda}+ e^{-\lambda} \sum_{n=1}^{\infty} \sum_{m=n}^{\infty} \frac{\lambda^n}{n!} (1-\kappa)^n \kappa^{m-n} \binom{m-1}{m-n} z^{m} \\
& =   e^{-\lambda}+ e^{-\lambda} \sum_{m=1}^{\infty} \sum_{n=1}^{m} \frac{\lambda^n}{n!} (1-\kappa)^n \kappa^{m-n} \binom{m-1}{n-1} z^{m} \\
& =  e^{-\lambda}+ e^{-\lambda} \sum_{m=1}^{\infty} \sum_{s=0}^{m-1} \frac{\lambda^{s+1}}{(s+1)!} (1-\kappa)^{s+1} \kappa^{m-s-1} \binom{m-1}{s} z^{m}.
\end{align*}
The probabilities $\p(X=m)$ are obtained from this results as the coefficients of the probability generating function by $z^m$, as $\Pcal(z) = \sum_{m=0}^{\infty} \p(X= m)z^m.$
\end{proof}

\begin{corollary}
Let $X$ be $\pds(\gamma,\lambda,\kappa$) random variable with $\gamma = 1$ and $\kappa>0$. Then the probability $\p(X=m)$ for $m\geq 1$  can be expressed in the following ways
\begin{equation*}
\p\left(X = m \right) = e^{-\lambda} \lambda (1-\kappa) \kappa^{m-1} \, {}_1F_1\left(1-m,2,\frac{\beta-1}{\beta}\lambda\right)
\end{equation*}
and
\begin{equation*}
\p\left(X = m \right) = e^{-\lambda} \lambda (1-\kappa) \kappa^{m-1} \, \frac{1}{m} L^{(1)}_{m-1}\left(\frac{\beta-1}{\beta}\lambda\right),
\end{equation*} 
where ${}_1F_1(a,b,z)$ is the Kummer confluent hypergeometric function and $L_n^{(\alpha)}(z)$ is the generalized Laguerre polynomial.
\end{corollary}
\begin{proof}
The first assertion follows directly from \eqref{pm_pds}. The second assertion follows from the relation between Laguerre polynomial and Kummer confluent hypergeometric function (see for example \cite[pp. 268]{erdelyiI}), stating that
$$L_n^{(\alpha)}(z) = \binom{n+\alpha}{n} {}_1F_1(-n,\alpha+1,z).$$
\end{proof}

% ################################################### LIMIT DISTRIBUTION ####################################################

\subsection{Continuous analogies}
Let us consider a~random variable $X^a = aX$, with $X \sim $ PDS($\gamma,\lambda,\kappa$) and $a>0$. Then $X^a$ takes values in $a \N_0 = \{0, a, 2a, \cdots\}$. We study the limit behaviour of $X^a$ as $a \to 0$ with $\kappa \to 1$.

\begin{theorem}
Let $X$ be a~positive discrete stable random variable with parameters $\gamma$, $\lambda$ and $\kappa$ and let $X^a = aX$  with $a>0$. Let $\kappa = 1-ac$. Then
\begin{equation*}
f^a(t) = \exp\left\{-\lambda\left(\frac{1-e^{\im at}}{1- \kappa e^{\im at}}\right)^{\gamma}\right\}
\longrightarrow \varphi(t) = \exp\left\{-\lambda\left(\frac{- \im t}{c - \im t}\right)^{\gamma}\right\}, \quad \text{as } a \to 0.
\end{equation*}
\end{theorem}

\begin{proof}
The limit characteristic function can be computed in a~straightforward way. We have
\begin{align*}
\frac{1 - e^{\im at}}{1-\kappa e^{\im at}} & = \frac{1 - e^{\im  at}}{1 - e^{\im at} + ac e^{\im at}} \approx \frac{-\im at}{-\im at + ac e^{\im at} },  \quad \text{as} \quad a\to 0. 
\end{align*}
Hence we have $$\varphi(t) = \lim_{a \to 0} \exp\left\{-\lambda\left(\frac{-\im at}{-\im at + ac e^{\im at} }\right)^{\gamma} \right\} = \exp\left\{-\lambda \left(\frac{-\im t}{-\im t + c}\right)^{\gamma}\right\}.$$
\end{proof}

Next we show that discrete stable distribution on $\N_0$ can be considered a~discrete analogy of stable distribution with index of stability $\alpha = \gamma$ and  skewness parameter $\beta = 1$. 
\begin{theorem}\label{th:pds_limit}
Let $X$ be a~positive discrete stable random variable with parameters $\gamma$, $\lambda$ and $\kappa$ and let $X^a = aX$  with $a>0$. Let $\lambda = b/a^{\gamma}$. Then
\begin{multline*}
f^a(t) = \exp\left\{-\lambda\left(\frac{1-e^{\im at}}{1- \kappa e^{\im at}}\right)^{\gamma}\right\} \\
\longrightarrow \varphi(t) = \exp\left\{-\sigma |t|^{\gamma}\left(1-\im \, \mathrm{sign}(t) \tan\left(\frac{\pi \gamma}{2}\right) \right) \right\}, \quad \text{as } a \to 0,
\end{multline*}
where $\sigma = \frac{b}{(1-\kappa)^{\gamma}}\cos\left(\frac{\pi \gamma}{2}\right).$
\end{theorem}
\begin{proof}
We have
\begin{align*}
\frac{1 - e^{\im at}}{1-\kappa e^{\im at}} & = \frac{1 - e^{\im  at}}{1 - \kappa + \kappa\left(1- e^{\im at}\right)} \approx \frac{-\im at}{(1-\kappa)- \kappa \im at}  \quad \text{as} \quad a \to 0. 
\end{align*}
Hence 
\begin{align*}
-\lambda\left(\frac{1-e^{\im at}}{1- \kappa e^{\im at}}\right)^{\gamma} & \approx - \frac{b}{a^{\gamma}} \left( \frac{-\im at}{(1-\kappa)- \kappa \im at}\right)^{\gamma}  \quad \text{as} \quad a \to 0  \\
& \to - \frac{b}{(1-\kappa)^{\gamma}} (-\im t)^{\gamma}   \quad \text{as} \quad a \to 0. 
\end{align*}
Finally we notice that $$(-\im t)^{\gamma} = |t|^{\gamma}(-\im \, \mathrm{sign}(t))^{\gamma} = |t|^{\gamma} \cos\left(\pi \gamma/2\right)(1-\im \, \mathrm{sign}(t) \tan\left(\pi \gamma/2\right)).$$
\end{proof}

%%%%%%%%%%%%%%%%%%%%%%%%%%%%%%%%%%%%%%%%%%%%%%%%%%%%%%%%%%%%%%%%%%%%%%%%%%%%%%%%%%%%%%%%%%%%%%%%%%%%%%%%%%%%%%%%%%%%%%%%%%%%%%%%%%%%%%%%%%%%%%%%%%%%%%%%%%%%%%%%%%%%%%%%%%%%%%%%%%%%%%%%%%%%%

\subsection{Asymptotic behaviour}
In this Subsection we show that the tails of discrete stable $\pds(\gamma,\lambda,\kappa)$ distribution are heavy with tail index $\gamma$.

\begin{proposition}
The discrete stable distribution $\pds(\gamma,\lambda,\kappa)$ belongs to the domain of normal attraction of $\alpha$-stable distribution with characteristic function $$g(t) = \exp\left\{-\frac{\lambda}{(1-\kappa)^\gamma}\cos\left(\pi \gamma/2\right)|t|^{\gamma}\left(1-\im \, \mathrm{sign}(t) \tan\left(\frac{\pi \gamma}{2}\right)\right)\right\}.$$
\end{proposition}
\begin{proof}
Let $X_1, X_2, \dots, X_n$ be i.i.d.~$\pds(\gamma,\lambda,\kappa)$ random variables with characteristic function $$f(t) = \exp\left\{-\lambda\left(\frac{1-e^{\im t}}{1 - \kappa e^{\im t}}\right)^{\gamma}\right\}.$$ Let us denote $S_n$ the normalized sum $$S_n = \frac{X_1 + X_2 + \dots + X_n}{n^{1/\gamma}} .$$ Then the characteristic function of $S_n$ is given as 
\begin{equation*}
\E\left[e^{\im t S_n}\right] = f^n\left(\frac{t}{n^{1/\gamma}}\right) = \exp\left\{-\lambda\left(\frac{1-e^{\im t/n^{1/\gamma}}}{1-\kappa e^{\im t/n^{1/\gamma}}}\right)^{\gamma}\right\}.
\end{equation*}
We use the Taylor expansion of $\exp$ to obtain 
\begin{align*}
\log \E\left[e^{\im t S_n}\right] & = -\lambda n \left(\frac{-\im t}{(1-\kappa) n^{1/\gamma}} + O(t^2/n^{2/\gamma}) \right)^{\gamma} \\
& = -\frac{\lambda}{(1-\kappa)^{\gamma}}(-\im t)^{\gamma} (1+O(n^{-2/\gamma}))^{\gamma},\quad \text{as} \quad n \to \infty.
\end{align*}
Hence $$g(t) = \lim_{n \to \infty} \E\left[e^{\im t S_n}\right] = \exp\left\{-\frac{\lambda}{(1-\kappa)^{\gamma}}(-\im t)^{\gamma}\right\}.$$
We can rewrite the exponent using $$(-\im t)^{\gamma} = |t|^{\gamma}(-\im \, \mathrm{sign}(t))^{\gamma} = |t|^{\gamma} \cos\left(\pi \gamma/2\right)(1-\im \, \mathrm{sign}(t) \tan\left(\pi \gamma/2\right)).$$
\end{proof}

%%%%%%%%%%%%%%%%%%%%%%%%%%%%%%%%%%%%%%%%%%%%%%%%%%%%%%%%%%%%%%%%%%%%%%%%%%%%%%%%%%%%%%%%%%%%%%%%%%%%%%%%%%%%%%%%%%%%%%%%%%%%%%%%%%%%%%%%%%%%%%%%%%%%%%%%%%%%%%%%%%%%%%%%%%%%%%%%%%%%%%%%%%%%%

%\subsection{Asymptotic expansion of probabilities}
%In this Subsection we study the asymptotic behaviour of the positive discrete stable distribution, i.e. the tails behaviour of the distribution. The asymptotic expansions of probabilities for $\kappa = 0$ were given in \cite{christoph}, where authors showed 
%$$RESULT.$$ In   
%

%%%%%%%%%%%%%%%%%%%%%%%%%%%%%%%%%%%%%%%%%%%%%%%%%%%%%%%%%%%%%%%%%%%%%%%%%%%%%%%%%%%%%%%%%%%%%%%%%%%%%%%%%%%%%%%%%%%%%%%%%%%%%%%%%%%%%%%%%%%%%%%%%%%%%%%%%%%%%%%%%%%%%%%%%%%%%%%%%%%%%%%%%%%%%
%%%%%%%%%%%%%%%%%%%%%%%%%%%%%%%%%%%%%%%%%%%%%%%%%%%%%%%%%%%%%%%%%%%%%%%%%%%%%%%%%%%%%%%%%%%%%%%%%%%%%%%%%%%%%%%%%%%%%%%%%%%%%%%%%%%%%%%%%%%%%%%%%%%%%%%%%%%%%%%%%%%%%%%%%%%%%%%%%%%%%%%%%%%%%

\section{Properties of discrete stable random variables}

In this Section we will study more into detail the discrete stable distribution in the limit sense with two-sided modified geometric thinning operator, as defined in Section \ref{section:ds1}. To remind the definition, an~integer-valued random variable $X$ is said to be discrete stable in the limit sense, if 
\begin{equation}\label{eq_defds2}
X \stackrel{d}{=} \lim_{n\to \infty}\sum_{i=1}^n \bar{X}_i(p_n), \quad \text{where} \quad \bar{X}_i(p_n) =  \sum_{j=1}^{X_i^+} \varepsilon_j^{(i)} - \sum_{j=1}^{X_i^-} \epsilon_j^{(i)},
\end{equation}
where $X_1, X_2, \dots$ are independent copies of $X$ and $\varepsilon_j^{(i)},\epsilon_j^{(i)}$ are i.i.d.~integer-valued random variables. Throughout this Section we will assume that the random variables $\varepsilon_j^{(i)},\epsilon_j^{(i)}$ come from two-sided modified geometric distribution $2\mathcal{G}(p,\kappa,m,q)$ with probability generating function $\Rcal$. We remind that the probability generating function $\Rcal$ is given as 
\begin{equation}\label{ds_thinning}
\Rcal(z) = S^{-1} \circ B_p \circ S^{(2)} (z),
\end{equation} where 
\begin{align*}
S(z) & = \frac{(1-\kappa) z^m}{1-\kappa z^m}, \\
S^{-1}(y) & = \left(\frac{y}{1-\kappa(1-y)}\right)^{\frac{1}{m}}, \\
B_p(z) & = 1- p + pz,
\intertext{and finally}
S^{(2)}(z) & = q S(z) + (1-q) S(z^{-1}).
\end{align*}

%\begin{definition}\label{ds}
%Let $X, X_1, X_2, \dots, X_n, \dots$ be i.i.d.~integer-valued random variables with values in $m\Z$, where $m \in \N$. Assume that for every $n \in \N$ there exist constants $0< a~< \kappa<1$ such that
%\begin{equation}\label{def_ds}
%\sum_{k=1}^{n} \bar{X}_k \stackrel{d}{\rightarrow} X, \quad \text{where} \quad \bar{X}_k = \sum_{i=1}^{X_k^+} \varepsilon_i^{(k)} - \sum_{j=1}^{X_k^-} \epsilon_j^{(k)},
%\end{equation}
%$\varepsilon_i^{(k)}$,$\epsilon_j^{(k)}$  are i.i.d.~$2\mathcal{G}(a,\kappa,q,m)$ random variables. Then we say that $X$ is \textit{discrete stable random variable}.
%\end{definition}

\begin{theorem}\label{th_pgf_ds}
An integer-valued random variable $X$ is discrete stable in the limit sense with two-sided modified geometric thinning operator, if and only if $\Rcal(z)$ takes form \eqref{ds_thinning} and the probability generating function $\Pcal(z)= \E z^X = \sum_{k=-\infty}^{\infty} \p(X=k) z^k$ takes form
\begin{multline}\label{pgf_ds}
\Pcal(z) = \exp\left\{-\lambda \left(\frac{1+\beta}{2}\right) \left(1-q \frac{(1-\kappa)z^m}{1-\kappa z^m} - (1-q) \frac{(1-\kappa)z^{-m}}{1-\kappa z^{-m}}\right)^{\gamma}  \right.\\ 
\left. -\lambda \left(\frac{1-\beta}{2}\right) \left(1-(1-q) \frac{(1-\kappa)z^m}{1-\kappa z^m} - q \frac{(1-\kappa)z^{-m}}{1-\kappa z^{-m}}\right)^{\gamma}   \right\} 
\end{multline}
with $\gamma \in (0,1], \; \lambda > 0, \; \kappa \in [0,1), \beta \in [-1,1], q \in [0,1].$ 
\end{theorem}
\begin{proof}
We have shown in Proposition \ref{prop_defDS2} that a~random variable $X$ is discrete stable in the limit sense if and only if 
$$\Pcal(z) = \lim_{n \to \infty} \left[\Pcal_0 + \Pcal_1(\Rcal(z)) + \Pcal_2(\Rcal(1/z))\right], $$ where $\Pcal_1$ is the generating function of the sequence $\{p_1,p_2,\dots\}$ with $p_k = \p(X=k)$ and $\Pcal_2$ is the generating function of the sequence $\{q_1,q_2,\dots\}$ with $q_k = \p(X=-k)$. Let us assume that $\Pcal_1$ and $\Pcal_2$ take the following form
\begin{equation}\label{pg_x1_ds}
\Pcal_{i}(z) = \Pcal_i(1) - \lambda_i \left(\frac{1-z^m}{1-\kappa z^m}\right)^{\gamma} + o\left(\left(\frac{1-z^m}{1-\kappa z^m}\right)^{\gamma}\right), \quad i=1,2,
\end{equation}
with $\gamma \in (0,1]$. We notice that $$\frac{1-z^m}{1-\kappa z^m} = 1-S(z).$$ This simplifies the computation, as $1-S\left(\Rcal(z)\right) = 1 - \left(1-p+pS^{(2)}(z)\right) = p\left(1-S^{(2)}(z)\right)$ and similarly for $1-S(\Rcal(1/z))$. 

%By dreadful computation we can show that 
%\begin{align*}
%\frac{1}{p} \frac{1-\Rcal(z)^m}{1-\kappa \Rcal(z)^m} & \to \frac{(1-z^m)(1-q+\kappa q - z^m(q + \kappa(1-q)))}{\kappa z^{2m} - z^m(\kappa^2+1)+\kappa} \quad \text{as } a~\to 0,\\
%\intertext{and}
%\frac{1}{p} \frac{1-\Rcal(z^{-1})^m}{1-\kappa \Rcal(z^{-1})^m} & \to \frac{(1-z^m)(q + \kappa(1-q)- z^m(1-q+\kappa q))}{\kappa z^{2m} - z^m(\kappa^2+1)+\kappa} \quad \text{as } a~\to 0.
%\end{align*}

We can now compute the limit
$$\Pcal(z) = \lim_{n \to \infty} \left[\Pcal_0 + \Pcal_{1}\left(\Rcal(z)\right) +  \Pcal_{2}\left(\Rcal(1/z)\right)\right]^n.$$ Let $p = n^{-1/\gamma}$. Then
\begin{align*}
\Pcal(z)  & =  \lim_{n \to \infty} \left[1-\lambda_1 \left(1-S(\Rcal(z))\right)^{\gamma} -\lambda_2 \left( 1-S(\Rcal(1/z))\right)^{\gamma}  \right]^n \\ 
& =  \lim_{n \to \infty} \left[1-\lambda_1 p^\gamma \left(1-S^{(2)}(z)\right)^{\gamma} -\lambda_2 p^{\gamma} \left( 1-S^{(2)}(1/z)\right)^{\gamma}  \right]^n \\ 
& = \exp\left\{-\lambda_1 \left(1-q \frac{(1-\kappa)z^m}{1-\kappa z^m} - (1-q) \frac{(1-\kappa)z^{-m}}{1-\kappa z^{-m}}\right)^{\gamma}  \right.\\ 
&\quad \quad \quad  \left.- \lambda_2 \left(1-q \frac{(1-\kappa)z^{-m}}{1-\kappa z^{-m}} - (1-q)\frac{(1-\kappa)z^m}{1-\kappa z^m}  \right)^{\gamma} \right\}.
\end{align*}
By setting $\lambda = \lambda_1 + \lambda_2$ and $\beta = \frac{\lambda_1-\lambda_2}{\lambda_1 + \lambda_2}$, we obtain the desired result.
\end{proof}

We will denote discrete stable distribution (and random variable) by DS$^m(\gamma,\beta,\lambda,q,\kappa)$. The parameter $m$ specifies the size of the lattice of the distribution. If we omit $m$ then it is understood that $m=1$. If $\kappa$ is omitted we will understand that $\kappa = 0$. If moreover $q$ is omitted we will understand that $q=1$. In this case the probability generating function \eqref{pgf_ds} reduces to $$\exp\left\{-\lambda \left(\frac{1+\beta}{2}\right) (1-z)^{\gamma} - \lambda \left(\frac{1-\beta}{2}\right) (1-1/z)^{\gamma} \right\}$$ which corresponds to the discrete stable distribution introduced in \cite{slamova2012}. In the case of $\beta = 1$ and $q = 1$, the $\ds(\gamma,1,\lambda,1,\kappa)$ random variable correspond to positive discrete stable random variable $\pds(\gamma,\lambda,\kappa)$.  %In Figure \ref{fig:ds} the probabilities of $\ds(\gamma,\beta,\lambda,1,\kappa)$ random variables are shown for different values of parameters. The probabilities were again obtained using the classical inverse Fourier transform theorem and the fast Fourier transform algorithm.

\begin{remark}
A discrete stable random variable $X \sim \ds(\gamma,\beta,\lambda,q,\kappa)$ is infinitely divisible, as for all $n \in \N$, $$X = Y_1 + Y_2 + \dots + Y_n, \quad \text{ where } \quad Y_i \sim \ds(\gamma,\beta,\lambda/n,q,\kappa), \; i=1,\dots, n.$$
\end{remark}

For the sake of simplicity we will denote 
\begin{align}
g(z) & =\left(1-q \frac{(1-\kappa)z^m}{1-\kappa z^m} - (1-q) \frac{(1-\kappa)z^{-m}}{1-\kappa z^{-m}}\right)^{\gamma},  \label{not:ds1} \\
h(z) & = g\left(z^{-1}\right) = \left(1-(1-q) \frac{(1-\kappa)z^m}{1-\kappa z^m} - q \frac{(1-\kappa)z^{-m}}{1-\kappa z^{-m}}\right)^{\gamma} . \label{not:ds2} 
\end{align}

Then the probability generating function of a~$\ds(\gamma,\beta,\lambda,q,\kappa)$ random variable can be written simply as
$$\Pcal(z) = \exp\left \{-\lambda  \left(\frac{1+\beta}{2} \right) g(z) - \lambda  \left(\frac{1-\beta}{2} \right) h(z)\right\}.$$

%%%%%%%%%%%%%%%%%%%%%%%%%%%%%%%%%%%%%%%%%%%%%%%%%%%%%%%%%%%%%%%%%%%%%%%%%%%%%%%%%%%%%%%%%%%%%%%%%%%%%%%%%%%%%%%%%%%%%%%%%%%%%%%%%%%%%%%%%%%%%%%%%%%%%%%%%%%%%%%%%%%%%%%%%%%%%%%%%%%%%%%%%%%%%

\subsection{Properties}
Discrete stable distribution shares many interesting properties with stable distributions. In this Subsection we show that analogies of Properties of stable distributions (see, for example, \cite{samorodnitsky}) hold also for discrete stable distributions.
\begin{property}
Let $X_1$ and $X_2$ be independent random variables with $X_i \sim \ds(\gamma,\beta_i, \lambda_i,q, \kappa)$, $i=1,2$. Then $X_1 + X_2 \sim \ds(\gamma,\beta, \lambda,q,\kappa)$, with $$\lambda = \lambda_1 + \lambda_2, \quad \beta = \frac{\beta_1 \lambda_1 + \beta_2 \lambda_2}{\lambda_1 + \lambda_2}.$$
\end{property}
\begin{proof}
Using the notation \eqref{not:ds1}--\eqref{not:ds2}, the probability generating function of $X_i$, $i=1,2$, is $$\Pcal_i(z) = \exp\left \{-\lambda_i  \left(\frac{1+\beta_i}{2} \right) g(z) - \lambda_i  \left(\frac{1-\beta_i}{2} \right) h(z)\right\}.$$  The probability generating function of $X_1 + X_2$ is a~product of the single probability generating functions. Therefore
\begin{align*}
\log \Pcal_{X_1 + X_2}(z) = & -\lambda_1 \left(\frac{1+\beta_1}{2} \right) g(z) - \lambda_1 \left(\frac{1-\beta_1}{2} \right) h(z) \\
& - \lambda_2 \left(\frac{1+\beta_2}{2} \right) g(z) - \lambda_2 \left(\frac{1-\beta_2}{2} \right) h(z) \\
= & - (\lambda_1 + \lambda_2) \frac{1}{2}\left(1+\frac{\lambda_1\beta_1 + \lambda_2\beta_2}{\lambda_1+\lambda_2}\right) g(z) \\
& -  (\lambda_1 + \lambda_2) \frac{1}{2}\left(1-\frac{\lambda_1\beta_1 + \lambda_2\beta_2}{\lambda_1+\lambda_2}\right) h(z) \\
= & -\lambda  \left(\frac{1+\beta}{2} \right) g(z) - \lambda  \left(\frac{1-\beta}{2} \right) h(z),
\end{align*}
where $\lambda = \lambda_1 + \lambda_2$ and $\beta = (\beta_1 \lambda_1 + \beta_2 \lambda_2)/(\lambda_1 + \lambda_2)$.
\end{proof}

\begin{property}
Let $X \sim \pds(\gamma,\lambda,\kappa)$. Let $a \in (0,1)$. Then $\tilde{X}(a) \sim \pds(\gamma, a^\gamma \lambda,\kappa)$.
\end{property}
\begin{proof}
The probability generating function of $\tilde{X}(a)$ is equal to $$\exp \left\{-\lambda \left( 1- S(\Qcal_a(z))\right)^{\gamma} \right\} = \exp \left\{-\lambda a^{\gamma} \left( 1- S(z)\right)^{\gamma} \right\}.$$
\end{proof}

\begin{property}
Let $X \sim \ds(\gamma,\beta,\lambda,q,\kappa)$. Then $-X \sim \ds(\gamma,-\beta,\lambda,q,\kappa)$.
\end{property}

\begin{proof}
This follows from the fact that $g(z^{-1}) = h(z)$, where we use the notation \eqref{not:ds1}--\eqref{not:ds2}. Then the probability generating function of $-X$ is given as $$\Pcal(z^{-1}) = \exp\left\{-\lambda \left(\frac{1+\beta}{2}\right) h(z) -\lambda \left(\frac{1-\beta}{2}\right) g(z) \right\},$$ and this is the probability generating function of $\ds(\gamma,-\beta,\lambda,q,\kappa)$.
\end{proof}

\begin{property}
Let $X \sim \ds(\gamma,\beta,\lambda,q,\kappa)$. Then $X$ is symmetric if and only if $q = 1/2$ or $\beta = 0$. 
\end{property}

\begin{proof}
A discrete random variable is symmetric if and only if $\Pcal(z) = \Pcal(z^{-1})$. Using the notation \eqref{not:ds1}--\eqref{not:ds2}, and the fact that $g(z^{-1}) = h(z)$, it follows that a~discrete stable random variable is symmetric if and only if 
\begin{equation*}
-\lambda \left(\frac{1+\beta}{2}\right) g(z) -\lambda \left(\frac{1-\beta}{2}\right) h(z) = -\lambda \left(\frac{1+\beta}{2}\right) h(z) -\lambda \left(\frac{1-\beta}{2}\right) g(z).
\end{equation*}
But this holds true if and only if $\beta = 0$ or $g(z) = h(z)$. The latter condition is satisfied only if $q = 1/2$.
\end{proof} 

\begin{property}
Let $X$ be $\ds(\gamma,\beta,\lambda,q,\kappa)$. Then there exist two i.i.d.~random variables $Y_1$ and $Y_2$ with common distribution $\ds(\gamma,1,\lambda,1,\kappa)$ such that 
$$X \stackrel{d}{=} \bar{Y}_1\left(\left(\frac{1+\beta}{2}\right)^{1/\gamma}\right) - \bar{Y}_2\left(\left(\frac{1-\beta}{2}\right)^{1/\gamma}\right).$$ 
\end{property}

\begin{proof}
Let $Y_1, Y_2 \sim \ds(\gamma,1,\lambda,1,\kappa)$. Their probability generating function is $$\Pcal(z) = \exp\left\{-\lambda\left(\frac{1-z}{1-\kappa z}\right)^{\gamma}\right\}.$$ Moreover, the probability generating function of $\bar{Y_i}(p)$ is obtained in closed form, as $Y_i$ are in fact positive discrete stable random variables. So we have 
$$\Pcal_{\bar{Y_i}(p)} = \Pcal(\Rcal_p(z))$$ Similarly as in the Proof of Theorem \ref{th_pgf_ds} we can compute that $$\Pcal(\Rcal_p(z)) =  \exp\left\{-\lambda \left(1-S(\Rcal_p(z))\right)^{\gamma}\right\} = \exp\left\{-\lambda p^{\gamma} \left(1- q \frac{(1-\kappa) z}{1-\kappa z} - (1-q)\frac{(1-\kappa)z^{-1}}{1-\kappa z^{-1}}\right)^{\gamma}\right\}.$$ The probability generating function of the difference $\bar{Y_1}(p_1) - \bar{Y_2}(p_2)$ is computed as 
$$\Pcal(\Rcal_{p_1}(z))\Pcal(\Rcal_{p_2}(1/z)).$$ Putting all together we obtain the desired result. 
\end{proof}

%%%%%%%%%%%%%%%%%%%%%%%%%%%%%%%%%%%%%%%%%%%%%%%%%%%%%%%%%%%%%%%%%%%%%%%%%%%%%%%%%%%%%%%%%%%%%%%%%%%%%%%%%%%%%%%%%%%%%%%%%%%%%%%%%%%%%%%%%%%%%%%%%%%%%%%%%%%%%%%%%%%%%%%%%%%%%%%%%%%%%%%%%%%%%

\subsection{Continuous analogies}
Let us consider a~random variable $X^a = aX$, with $X \sim $ DS($\gamma,\beta,\lambda,q, \kappa)$ and $a>0$. Then $X^a$ takes values in $a \Z = \{0, \pm a, \pm 2a, \cdots\}$. We show that the limit distribution of $X^a$ is $\alpha$-stable distribution with index of stability $\gamma$ and skewness $\beta$. We study the limit behaviour of $X^a$ as $a \to 0$ and $q \to 1/2$.

\begin{theorem}
Let $X$ be a~discrete stable random variable with parameters $\gamma$, $\beta$, $\lambda$, $q$ and $\kappa = 0$. Let $X^a = aX$  with $a>0$ and let  $2q - 1 \approx a$ as $a \to 0$. Then
\begin{multline*}
f^a(t) = \exp\left\{-\lambda \left(\frac{1+\beta}{2}\right) \left(1-qe^{\im at} - (1-q)e^{-\im at}\right)^{\gamma}- \right.\\
- \left. \lambda \left(\frac{1-\beta}{2}\right) \left(1-q e^{-\im at} - (1-q)e^{\im at}\right)^{\gamma}\right\} \\
\longrightarrow \varphi(t) = \exp\left\{-\lambda \cos \frac{\pi \gamma}{2} |t|^{\gamma} \left(1-\im \beta \mathrm{sign}(t) \tan \frac{\pi \gamma}{2}\right)\right\}, \quad \text{as } a\to 0.
\end{multline*}
\end{theorem}

\begin{proof}
We may rewrite the characteristic exponent of $f^a(t)$ as
\begin{align*}
\log f^a(t) & \approx -\lambda \left(\frac{1+\beta}{2}\right) \left((2q-1)(-\im at)\right)^{\gamma}-\lambda \left(\frac{1-\beta}{2}\right) \left((2q-1)(\im at)\right)^{\gamma}, \quad \text{as} \quad a~\to 0 \\
\intertext{and because $q \approx (1+a)/2$ we have}
& \approx -\lambda \left(\frac{1+\beta}{2}\right) (-\im t)^{\gamma} -\lambda \left(\frac{1-\beta}{2}\right) (\im t)^{\gamma}.
\end{align*}
To complete the proof it is enough to notice that $(-\im t)^{\gamma} = |t|^{\gamma}\left(\cos\tfrac{\pi \gamma}{2} - \im \sin \tfrac{\pi \gamma}{2} \right)$ and $(\im t)^{\gamma} = |t|^{\gamma}\left(\cos\tfrac{\pi \gamma}{2} + \im \sin \tfrac{\pi \gamma}{2} \right)$. 
\end{proof}

\begin{remark}
It can be shown that the case of $\kappa > 0$ leads to a~similar result, the limit distribution is again $\alpha$-stable with index of stability $\gamma$ and skewness $\beta$.
\end{remark} 

%Now we are interested in the limit behaviour of $X^a$ as $a \to 0$ and $\kappa \to 1$. 
%
%
%\begin{theorem}
%Let $X$ be a~discrete stable random variable with parameters $\gamma$, $\beta$ $\lambda$, $q$ and $\kappa$ and let $X^a = aX$  with $a>0$. Let $\kappa = 1-ac$ and $2q - 1 \approx a$ as $a \to 0$. Then
%\begin{align*}
%f^a(t) \longrightarrow \varphi(t) = \exp\left\{-\lambda \left(\frac{t^2}{t^2 + c^2}\right)^{\gamma}\right\}, \quad \text{as } a~\to 0.
%\end{align*}
%\end{theorem}

\begin{proof}

\end{proof}

%%%%%%%%%%%%%%%%%%%%%%%%%%%%%%%%%%%%%%%%%%%%%%%%%%%%%%%%%%%%%%%%%%%%%%%%%%%%%%%%%%%%%%%%%%%%%%%%%%%%%%%%%%%%%%%%%%%%%%%%%%%%%%%%%%%%%%%%%%%%%%%%%%%%%%%%%%%%%%%%%%%%%%%%%%%%%%%%%%%%%%%%%%
%%%%%%%%%%%%%%%%%%%%%%%%%%%%%%%%%%%%%%%%%%%%%%%%%%%%%%%%%%%%%%%%%%%%%%%%%%%%%%%%%%%%%%%%%%%%%%%%%%%%%%%%%%%%%%%%%%%%%%%%%%%%%%%%%%%%%%%%%%%%%%%%%%%%%%%%%%%%%%%%%%%%%%%%%%%%%%%%%%%%%%%%%%

\section{Properties of symmetric discrete stable random variables}

In the previous Section we studied the general case of discrete stable distribution in the limit sense. The symmetric version of such distribution is special case with interesting properties and we will therefore study it more into details in this Section.  The symmetric discrete stable distribution in the limit sense is obtained by considering the symmetric two-sided modified geometric thinning operator $2\mathcal{G}(a,\kappa,\tfrac{1}{2},m)$.

%\begin{definition}\label{sds}
%Let $X, X_1, X_2, \dots, X_n, \dots$ be i.i.d.~symmetric integer-valued random variables with values in $m\Z$, where $m \in \N$. Assume that for every $n \in \N$ there exist constants $0< a~< \kappa<1$ such that
%\begin{equation}\label{def_sds}
%\sum_{k=1}^{n} \bar{X}_k \stackrel{d}{\rightarrow} X, \quad \text{where} \quad \bar{X}_k = \sum_{i=1}^{X_k^+} \varepsilon_i^{(k)} - \sum_{j=1}^{X_k^-} \epsilon_j^{(k)},
%\end{equation}
%$\varepsilon_i^{(k)}$,$\epsilon_j^{(k)}$  are i.i.d.~$2\mathcal{G}(a,\kappa,\tfrac{1}{2},m)$ random variables. Then we say that $X$ is \textit{symmetric discrete stable random variable}.
%\end{definition}

\begin{theorem}
A symmetric integer-valued random variable $X$ is symmetric discrete stable with symmetric two-sided $\mathcal{G}$~thinning operator if and only if the thinning operator takes form \eqref{ds_thinning} with $q=1/2$ and the probability generating function $\Pcal(z)= \E z^X \sum_{k=-\infty}^{\infty} \p(X=k) z^k$ takes form
\begin{equation}\label{pgf_sds}
\Pcal(z) = \exp\left\{-\lambda\left(1-\frac{1-\kappa}{2}\left(\frac{z^m}{1-\kappa z^m}+\frac{z^{-m}}{1-\kappa z^{-m}}\right)\right)^{\gamma}\right\}
\end{equation}
with parameters  $\gamma \in (0,1]$, $\lambda > 0$, $\kappa \in [0,1)$ and $m \in \N$. 
\end{theorem}
\begin{proof}
The proof follows from the proof of Theorem \ref{th_pgf_ds}. In the symmetric case we have $\Pcal_1(z) = \Pcal_2(z)$, therefore $\lambda_1 = \lambda_2$ and moreover $q=1/2$. The probability generating function \eqref{pgf_ds} thus reduces to \eqref{pgf_sds}.
\end{proof}

We will denote symmetric discrete stable distribution (and also random variable) by $\sds^m(\gamma,\lambda,\kappa)$. In case when $m$ is omitted we will understand that $m=1$. If $\kappa$ is omitted we will understand that $\kappa = 0$, in which case the symmetric discrete stable distribution reduces to the symmetric discrete stable distribution as it was introduced in \cite{slamova2012}. %In Figure \ref{fig:sds} the probabilities of $\sds(\gamma,\lambda,\kappa)$ random variables are shown for different values of parameters. The probabilities were again obtained using the classical inverse Fourier transform theorem and the fast Fourier transform algorithm.

The characteristic function is given as 
$$f(t) = \exp\left\{-\lambda\left(1-(1-\kappa)\frac{\cos(t m)-\kappa}{\kappa^2 - 2\kappa \cos(t m)+1}\right)^{\gamma}\right\}.$$

The case of $\gamma = 1$ is a~special one as it leads to a~distribution with finite variance and exponential tails.

% #################################################### MODIFIED SDS ####################################################

\subsection{Characterizations}

\begin{theorem}\label{char_theorem_sds}
Let $\gamma, \gamma' \in (0,1]$ and assume that $\gamma' \leq \gamma$. Let $\s_\gamma$ be a~$\gamma$-stable random variable with Laplace transform $\exp\{-u^{\gamma}\}$. Then 
$$\sds(\gamma',\lambda,\kappa) \stackrel{d}{=} \sds\left(\gamma'/\gamma, \lambda^{1/\gamma} \s_{\gamma},\kappa\right).$$
\end{theorem}

\begin{proof}
The proof of the Theorem is done in the same way as the proof of Theorem \ref{char_theorem_pds}.
\end{proof}

\begin{corollary}\label{sds_simul}
Let $Y, Y_1, Y_2, \dots$ be a~sequence of i.i.d. random variables with two-sided geometric distribution, $\p(Y= \pm n) = \frac{1}{2}(1-\kappa) \kappa^{n-1}, n \geq 1.$ Let $N$ be a~random variable, independent of the sequence $Y_1,Y_2, \dots$, with Poisson distribution with random intensity $\lambda^{-1/\gamma} \s_{\gamma}$, where $\s_{\gamma}$ is a~$\gamma$-stable random variable  with Laplace transform $\exp\{-u^{\gamma}\}$. A random variable $X$ is symmetric discrete stable $\sds(\gamma,\lambda,\kappa)$ if and only if
$$X \stackrel{d}{=} \sum_{j=1}^N Y_j.$$
\end{corollary}
\begin{proof}
Let $X = \sum_{j=1}^N Y_j$. Then $X$ is a~compound Poisson random variable with random intensity $\lambda^{1/\gamma} \s_{\gamma}$ and jumps $Y_1, Y_2, \dots$ with characteristic function $$g(t) = \frac{1}{2}\frac{(1-\kappa)e^{\im t}}{1-\kappa e^{\im t}}+\frac{1}{2}\frac{(1-\kappa)e^{-\im t}}{1-\kappa e^{-\im t}}.$$ The characteristic function of a~compound Poisson random variable with intensity $\tau$ and characteristic function of jumps $h(t)$ is $\exp\{-\tau(1-h(t))\}.$ Therefore $X$ is in fact $\sds(1,\lambda^{1/\gamma}\s_{\gamma},\kappa)$.  We thus obtain the result from the previous Theorem \ref{char_theorem_sds} with $\gamma'= \gamma$.
\end{proof}

% #################################################### Probabilities ####################################################

\subsection{Probabilities}

\begin{theorem}
Let $X$ be $\sds(\gamma,\lambda)$ random variable. Then
\begin{equation*}
\p(X=k) = \sum_{i = |k|}^{\infty} \sum_{j = 0}^{\infty} (-1)^{i+j} \binom{\gamma j}{i} \frac{\lambda^j}{j!} \frac{1}{2^i} \binom{i}{\frac{i+k}{2}}, \quad k\in \Z.
\end{equation*}
In case $\gamma = 1$ this simplifies to 
\begin{equation*}
\p(X=k) = e^{-\lambda} I_k(\lambda), \quad k\in \Z.
\end{equation*}
where $I_k$ is the modified Bessel function of the first kind.
\end{theorem}
\begin{proof}
The generating function of a~discrete random variable taking values in $\Z$ is a~power series, with coefficients equal to probabilities, i.e. $$\Pcal_X(z) = \sum_{k=-\infty}^{\infty} \p(X=k) z^k.$$ (Note that this series converges only for $\varepsilon< |z| \leq 1$). Thus expanding \eqref{pgf_sds} with $\kappa=0$ into a~power series we obtain the probabilities. We use Taylor expansion of exponential function, binomial expansion and interchange of sums.
\begin{align*}
\exp\left\{-\lambda \left[1-\frac{1}{2}\left(z + \frac{1}{z}\right)\right]^{\gamma}\right\} 
& = \sum_{j=0}^{\infty} \sum_{i=0}^{\infty} \sum_{l=0}^i (-1)^{i+j} \binom{\gamma j}{i} \binom{i}{l} \frac{\lambda^j}{j!} \frac{1}{2^i} z^{2l-i} = \\
\intertext{change of notation $k = 2l- i$ and interchange of sums}
& = \sum_{k=-\infty}^{\infty} \sum_{i=|k|}^{\infty} \sum_{j=0}^{\infty} (-1)^{i+j} \binom{\gamma j}{i} \binom{i}{\frac{i+k}{2}} \frac{\lambda^j}{j!} \frac{1}{2^i} z^{k}.
\end{align*}
From this the first result follows. Taking $\gamma = 1$ the first binomial coefficient $\binom{j}{i}$ turns 0 for $j<i$ and we have, for $k \geq 0$,
\begin{align*}
\p(X=k) & = \sum_{i=k}^{\infty} \sum_{j=i}^{\infty} (-1)^{i+j} \binom{j}{i} \binom{i}{\frac{i+k}{2}} \frac{\lambda^j}{j!} \frac{1}{2^i} = \\
& = e^{-\lambda} \sum_{l=0}^{\infty} (\lambda/2)^{k+2l} \frac{1}{\Gamma(l+1) \Gamma(l+k+1)} = \\
& = e^{-\lambda} I_k(\lambda). 
\end{align*}
\end{proof}

% #################################################### Limit distributions ####################################################

\subsection{Continuous analogies}
Let us consider a~case of random variable $X^a = aX$, with $X \sim $ SDS($\gamma,\lambda,\kappa$) and $a>0$. Then $X^a$ takes values in $a \Z = \{0, \pm a, \pm 2a, \cdots\}$. We study the limit behaviour of $X^a$ as $a \to 0$ with $\kappa \to 1$.

\begin{theorem}
Let $X$ be a~symmetric discrete stable random variable with parameters $\gamma$, $\lambda$ and $\kappa$ and let $X^a = aX$  with $a>0$. Let $\kappa = 1-ac$. Then
\begin{multline*}
f^a(t) = \exp\left\{-\lambda\left(1-(1-\kappa)\frac{\cos(at)-\kappa}{\kappa^2 - 2\kappa \cos(at)+1}\right)^{\gamma}\right\} \\
\longrightarrow \varphi(t) = \exp\left\{-\lambda \left(\frac{t^2}{t^2 + c^2}\right)^{\gamma}\right\}, \quad \text{as } a~\to 0.
\end{multline*}
\end{theorem}

\begin{proof}
The limit characteristic function can be computed in a~straightforward way. We have

\begin{align*}
\left(1-(1-\kappa)\frac{\cos(at) - \kappa}{\kappa^2 - 2\kappa \cos(at) + 1}\right) & = \left(1 + ac \frac{1-\cos(at) - ac}{2(1-ac)(1-\cos(at))+a^2 c^2} \right) \\
& \approx \left(1 + \frac{ac t^2/2 - c^2}{t^2 - a~c t^2 + c^2} \right) \quad \text{as} \quad a~\to 0 
\end{align*}
Hence we have $$\varphi(t) = \lim_{a \to 0} \exp\left\{-\lambda\left(1 + \frac{ac t^2/2 - c^2}{t^2 - ac t^2 + c^2}\right)^{\gamma} \right\} = \exp\left\{-\lambda \left(\frac{t^2}{t^2 + c^2}\right)^{\gamma}\right\}.$$

\end{proof}

Next we show that symmetric discrete stable is a~discrete analogy of symmetric stable distribution with index of stability $\alpha = 2 \gamma$. 
\begin{theorem}
Let $X$ be a~symmetric discrete stable random variable with parameters $\gamma$, $\lambda$ and $\kappa$ and let $X^a = aX$  with $a>0$. Let $\lambda = b/a^{2 \gamma}$. Then
\begin{equation*}
f^a(t) = \exp\left\{-\lambda\left(1-(1-\kappa)\frac{\cos(at)-\kappa}{\kappa^2 - 2\kappa \cos(at)+1}\right)^{\gamma}\right\} 
\longrightarrow \varphi(t) = \exp\left\{-\sigma |t|^{2 \gamma} \right\}, \quad \text{as } a~\to 0,
\end{equation*}
where $\sigma = \frac{b}{2^{\gamma}}\frac{(1+\kappa)^{\gamma}}{(1-\kappa)^{2\gamma}}.$
\end{theorem}
\begin{proof}
We have
\begin{align*}
1-(1-\kappa)\frac{\cos(at)-\kappa}{\kappa^2 - 2\kappa \cos(at)+1} & = (1+\kappa)\frac{1 -\cos(at)}{\kappa^2 - 2\kappa \cos(at)+ 1} \\
&  \approx \frac{(1+\kappa)}{2}\frac{a^2 t^2}{(1-\kappa)^2 + \kappa a^2 t^2} \quad \text{as} \quad a~\to 0 
\end{align*}
Hence 
\begin{align*}
-\lambda\left(1-(1-\kappa)\frac{\cos(at)-\kappa}{\kappa^2 - 2\kappa \cos(at)+1}\right)^{\gamma} & \approx - \frac{b}{a^{2\gamma}} \left( \frac{(1+\kappa)}{2}\frac{a^2 t^2}{(1-\kappa)^2 + \kappa a^2 t^2}\right)^{\gamma}  \quad \text{as} \quad a~\to 0  \\
& \to - \frac{b}{2^{\gamma}}\frac{(1+\kappa)^{\gamma}}{(1-\kappa)^{2 \gamma}} |t|^{2 \gamma}   \quad \text{as} \quad a~\to 0. 
\end{align*}
\end{proof}

% #################################################### MOMENTS ####################################################

\subsection{Moments}
In this Subsection we give a~formula for factorial moments of $\sds(1,\lambda,\kappa)$ distribution and show that fractional moments of SDS$(\gamma,\lambda,\kappa)$ of non-integer order up to $2 \gamma$ exists.
  
\begin{theorem}
Let $X$ be $\sds(\gamma,\lambda,\kappa)$ random variable with $\gamma = 1$ and $\kappa>0$. Then the $n$-th factorial moment can be computed using the following formula
\begin{equation}
\E\left[(X)_n\right] = \frac{1}{(1-\kappa)^n} \sum_{k=1}^{n} \frac{\lambda^k}{2^k} B_{n,k}\left(0,2!(\kappa - 1), 3! (\kappa^2 + 1), \dots , (n-k+1)! (\kappa^{n-k}-(-1)^{n-k+1}) \right),
\end{equation} 
where $B_{n,k}$ is the Bell's polynomial \eqref{bell}.
\end{theorem}
\begin{proof}
The proof is analogous to the proof of Theorem \ref{pds_moments} and therefore is omitted.
\end{proof}

\begin{theorem}
Let $X \sim \sds(\gamma,\lambda,\kappa)$ with $0<\gamma<1$. Then
\begin{align*}
\E|X|^r < \infty, & \quad \text{for any }\quad 0<r<2\gamma,\\
\E|X|^r = \infty, & \quad \text{for any }\quad r \geq 2\gamma.
\end{align*}
\end{theorem}
\begin{proof}
The moments of non-integer order $\E|X|^r$ for any $0<r<2$ can be computed using the following formula (see for example \cite[Lemma 2.2]{klebanov}):
$$\E|X|^r  = c_r \int_0^{\infty} (1-\mathrm{Re}(f(t))) \frac{\dd t}{t^{r+1}},$$ with $$c_r = -\frac{r}{\Gamma(1-r)\cos(\pi r/2)}$$
and where $f(t)$ is the characteristic function of the distribution of $X$. Since SDS is a~symmetric distribution, the characteristic function of $X$ is real, and equal to $$f(t) = \exp\left\{-\lambda\left(\frac{\big(1-\cos(t)\big)(1+\kappa)}{\kappa^2 - 2\kappa \cos(t)+1}\right)^{\gamma}\right\}.$$ We may thus compute the moments.
\begin{align*}
\E|X|^r & =  c_r \int_0^{\infty} \left[1-\exp\left\{-\lambda\left(\frac{\big(1-\cos(t)\big)(1+\kappa)}{\kappa^2 - 2\kappa \cos(t)+1}\right)^{\gamma}\right\}\right] \frac{\dd t}{t^{1+r}} \\ 
& =  c_r \int_0^1  \left[1-\exp\left\{-\lambda\left(\frac{\big(1-\cos(t)\big)(1+\kappa)}{\kappa^2 - 2\kappa \cos(t)+1}\right)^{\gamma}\right\}\right] \frac{\dd t}{t^{1+r}} \\
& \quad + c_r \int_1^{\infty} \left[1-\exp\left\{-\lambda\left(\frac{\big(1-\cos(t)\big)(1+\kappa)}{\kappa^2 - 2\kappa \cos(t)+1}\right)^{\gamma}\right\}\right] \frac{\dd t}{t^{1+r}}.
\end{align*}
Using the limit comparison test we see that the first integral converges for $r<2\gamma$ and diverges for $r \geq 2 \gamma$, and the second integral converges for all $r>0$.
\end{proof}

% #################################################### ASYMPTOTIC behaviour ####################################################

\subsection{Asymptotic behaviour}
In this Subsection we show that the tails of symmetric discrete stable $\sds(\gamma,\lambda,\kappa)$ distribution are indeed heavy with tail index $2 \gamma$.

\begin{proposition}
The symmetric discrete stable distribution $\sds(\gamma,\lambda,\kappa)$ belongs to the domain of normal attraction of symmetric $\alpha$-stable distribution with characteristic function $$g(t) = \exp\left\{-\frac{\lambda}{2^{\gamma}}\frac{(1+\kappa)^{\gamma}}{(1-\kappa)^{2\gamma}} |t|^{2\gamma}\right\}.$$
\end{proposition}
\begin{proof}
Let $X_1, X_2, \dots, X_n$ be i.i.d.~$\sds(\gamma,\lambda,\kappa)$ random variables with characteristic function $$f(t) = \exp\left\{-\lambda\left(\frac{\big(1-\cos(t)\big)(1+\kappa)}{\kappa^2 - 2\kappa \cos(t)+1}\right)^{\gamma}\right\}.$$ Let us denote $S_n$ the normalized sum $$S_n = \frac{X_1 + X_2 + \dots + X_n}{n^{1/2 \gamma}} .$$ Then the characteristic function of $S_n$ is given as 
\begin{equation*}
\E\left[e^{\im t S_n}\right] = f^n\left(\frac{t}{n^{1/2\gamma}}\right) = \exp\left\{-\lambda\left(\frac{\big(1-\cos(t/n^{1/2\gamma})\big)(1+\kappa)}{\kappa^2 - 2\kappa \cos(t/n^{1/2 \gamma})+1}\right)^{\gamma}\right\}.
\end{equation*}
We use the Taylor expansion of $\cos$ to obtain 
\begin{align*}
\log \E\left[e^{\im t S_n}\right] & = -\lambda n \left(\frac{t^2}{2 n^{1/\gamma}}\frac{1+\kappa}{(1-\kappa)^2} + O(n^{-3/2\gamma}) \right)^{\gamma} \\
& = -\frac{\lambda}{2^{\gamma}} \frac{(1+\kappa)^{\gamma}}{(1-\kappa)^{2\gamma}}|t|^{2 \gamma} (1+O(n^{-3/2\gamma}))^{\gamma},\quad \text{as} \quad n \to \infty.
\end{align*}
Hence $$g(t) = \lim_{n \to \infty} \E\left[e^{\im t S_n}\right] = \exp\left\{-\frac{\lambda}{2^{\gamma}}\frac{(1+\kappa)^{\gamma}}{(1-\kappa)^{2\gamma}}|t|^{2 \gamma}\right\}.$$
\end{proof}

\begin{theorem}
Let $X \sim \sds(\gamma, \lambda,\kappa)$ with $0<\gamma<1$. Then 
\begin{equation}
\lim_{x \to \infty} x^{2 \gamma} \p(|X| > x) = \left\{
\begin{array}{ll} 
\frac{\lambda}{2^{\gamma}}\frac{(1+\kappa)^{\gamma}}{(1-\kappa)^{2\gamma}}\frac{1}{\Gamma(1-2\gamma)\cos(\pi \gamma)}, & \text{if  } \gamma \ne \frac{1}{2}, \\ 
\frac{\lambda}{2^{\gamma}}\frac{(1+\kappa)^{\gamma}}{(1-\kappa)^{2\gamma}}\frac{2}{\pi}, & \text{if  } \gamma = \frac{1}{2}. 
\end{array}\right.
\end{equation}
\end{theorem}
\begin{proof}
We apply \cite[Theorem 2.6.7.]{ibragimov}: $\sds(\gamma,\lambda,\kappa)$ distribution belongs to the domain of normal attraction of $\s(\alpha,\beta,c,\mu)$ with $\alpha = 2\gamma$, $\beta = 0$, $c = \lambda/2^{\gamma}(1+\kappa)^{\gamma}(1-\kappa)^{-2\gamma}$ and $\mu=0$, hence the tail functions of $\sds(\gamma, \lambda,\kappa)$ are given as 
\begin{equation*}
\begin{array}{rll}
F(x) & = (c_1 + \alpha_1(x))|x|^{-\alpha}, &  \text{for} \quad  x<0, \\
1 - F(x) & = (c_2  + \alpha_2(x))x^{-\alpha}, &\text{for} \quad  x>0, 
\end{array}
\end{equation*}
where $\alpha_i(x) \to 0$ as $|x| \to \infty$.
The constants $c_1, c_2$ satisfy following conditions:
\begin{align*}
\beta & = (c_1 - c_2)/(c_1+c_2),\\
c & = 
\left\{\begin{array}{ll}
\Gamma(1-\alpha)(c_1+c_2)\cos(\pi \alpha/2), & \text{if} \quad \alpha \ne 1,\\
\frac{\pi}{2}(c_1 + c_2), & \text{if} \quad \alpha = 1.
\end{array}\right.
\end{align*}
We can easily see that for $\alpha \ne 1$ we have $$c_1 = c_2 = \frac{1}{2}\frac{\lambda}{2^{\gamma}}\frac{(1+\kappa)^{\gamma}}{(1-\kappa)^{2\gamma}}\frac{1}{\Gamma(1-2\gamma)\cos(\pi \gamma)},$$ and for $\alpha = 1$ we have $$c_1 = c_2 = \frac{\lambda}{2^{\gamma}}\frac{(1+\kappa)^{\gamma}}{(1-\kappa)^{2\gamma}}\frac{1}{\pi}.$$ Hence 
\begin{align*}
\lim_{x \to \infty} x^{2 \gamma} \p(|X|>x) & = \lim_{x \to \infty} x^{2 \gamma} (F(-x) + 1- F(x)) \\
& = \lim_{x \to \infty} x^{2 \gamma} \left[(c_1 + \alpha_1(-x))x^{-2 \gamma} + (c_2 + \alpha_2(x))x^{-2 \gamma} \right]  \\
& = 2 c_1.
\end{align*} 
\end{proof}

\subsection{Asymptotic expansion of probabilities}
In this Subsection we give an asymptotic expansion of the probabilities of the symmetric discrete stable distribution with $\kappa = 0$. The following result is an adaptation of the approach used in \cite{christoph} for positive discrete stable random variables. 
\begin{theorem}
Let $X \sim \sds(\gamma,\lambda)$, with $0<\gamma<1$. Then for any fixed integer $m$ and $n \to \infty$ 
\begin{equation}\label{pxn}
\p(X=n) = \frac{2^{-n}}{\pi} \sum_{j=1}^m \frac{(-1)^{j+1}}{j!}\lambda^j  \sin(\gamma j \pi) B(\gamma j +1,n-\gamma j) + O(n^{-\gamma(m+1)-1}),  \\
\end{equation}
where $B(x,y) = \Gamma(x)\Gamma(y)/\Gamma(x+y)$ is the Beta function. Moreover
\begin{equation}\label{pxn2}
\p(X=n) = \frac{2^{-n}}{\pi} \sum_{j=1}^{[(\gamma+1)/\gamma]} \frac{(-1)^{j+1}}{j!}\lambda^j \Gamma(\gamma j + 1) \sin(\gamma j \pi) n^{-\gamma j - 1} + O(n^{-\gamma-2}). 
\end{equation}
\end{theorem}
\begin{proof}
Using the stochastic representation of $\sds(\gamma,\lambda)$ random variable as a~compound Poisson random variable with random intensity (\cite{slamovaMME}) we have
\begin{align*}
\p(X=n) = \int_0^{\infty} e^{-s} I_n(s) p_{\gamma}^{\lambda}(s) \dd s,
\end{align*}
where $I_n(s)$ is the modified Bessel function of the first kind and $p_{\gamma}^{\lambda}(s)$ is the density function of the random variable $S_{\gamma}^{\lambda}$ with characteristic function $$g(t) = \exp\left\{-\lambda |t|^{\gamma} \exp(-\im \, \mathrm{sgn}(t) \gamma \pi/2)\right\}.$$
The density function $p_{\gamma}^{\lambda}(s)$ has the following series representation (\cite{christoph_wolf}):
\begin{equation}\label{pgammalambda}
p_{\gamma}^{\lambda}(s) = \frac{1}{\pi}\sum_{j=1}^m \frac{(-1)^{j+1}}{j!}\lambda^j \Gamma(\gamma j + 1) \sin(\gamma j \pi) s^{-\gamma j - 1} + A_m(s),
\end{equation} 
for any $m\geq 0$, where $A_m(s) = O(s^{-\gamma(m+1)-1})$ as $s \to \infty$. We may compute the probability as 
\begin{align*}
\p(X=n) & = \frac{1}{\pi} \sum_{j=1}^m \frac{(-1)^{j+1}}{j!}\lambda^j \Gamma(\gamma j + 1) \sin(\gamma j \pi) \int_{0}^{\infty} e^{-s} I_n(s) s^{-\gamma j - 1} \dd s + \int_0^{\infty} e^{-s} I_n(s) A_m(s) \dd s.
\end{align*} 
We approximate the modified Bessel function $I_n(s)$ by the first term of its infinite series representation $\Gamma(n+1)^{-1}(s/2)^n$. Then the first integral turns into $$\int_{0}^{\infty} e^{-s} I_n(s) s^{-\gamma j - 1} \dd s \approx \frac{1}{2^n} \frac{\Gamma(n-\gamma j)}{\Gamma(n+1)}, \quad \text{as} \quad n \to \infty.$$ The remainder term is obtained by computing the integral with $j=m+1$ and by approximating the ratio of two Gamma functions for large $n$ using the Stirling's formula
\begin{equation}\label{stirling}
\frac{\Gamma(n-\gamma j)}{\Gamma(n+1)} = n^{-\gamma j} \left(n^{-1} + O\left(n^{-2}\right)\right), \text{ as } n \to \infty.
\end{equation}

If we set $m = \left[(\gamma+1)/\gamma\right]$ and apply \eqref{stirling} on all terms in \eqref{pxn}, we obtain \eqref{pxn2}.

\end{proof}

%%%%%%%%%%%%%%%%%%%%%%%%%%%%%%%%%%%%%%%%%%%%%%%%%%%%%%%%%%%%%%%%%%%%%%%%%%%%%%%%%%%%%%%%%%%%%%%%%%%%%%%%%%%%%%%%%%%%%%%%%%%%%%%%%%%%%%%%%%%%%%%%%%%%%%%%%%%%%%%%%%%%%%%%%%%%%%%%%%%%%%%%%%%%%
%%%%%%%%%%%%%%%%%%%%%%%%%%%%%%%%%%%%%%%%%%%%%%%%%%%%%%%%%%%%%%%%%%%%%%%%%%%%%%%%%%%%%%%%%%%%%%%%%%%%%%%%%%%%%%%%%%%%%%%%%%%%%%%%%%%%%%%%%%%%%%%%%%%%%%%%%%%%%%%%%%%%%%%%%%%%%%%%%%%%%%%%%%%%%
%%%%%%%%%%%%%%%%%%%%%%%%%%%%%%%%%%%%%%%%%%%%%%%%%%%%%%%%%%%%%%%%%%%%%%%%%%%%%%%%%%%%%%%%%%%%%%%%%%%%%%%%%%%%%%%%%%%%%%%%%%%%%%%%%%%%%%%%%%%%%%%%%%%%%%%%%%%%%%%%%%%%%%%%%%%%%%%%%%%%%%%%%%%%%
%%%%%%%%%%%%%%%%%%%%%%%%%%%%%%%%%%%%%%%%%%%%%%%%%%%%%%%%%%%%%%%%%%%%%%%%%%%%%%%%%%%%%%%%%%%%%%%%%%%%%%%%%%%%%%%%%%%%%%%%%%%%%%%%%%%%%%%%%%%%%%%%%%%%%%%%%%%%%%%%%%%%%%%%%%%%%%%%%%%%%%%%%%%%%

\section{Properties of positive discrete stable random variables with thinning operator of Chebyshev type}\label{sec:pdsT}

The $\mathcal{G}$ thinning operator (of geometric type) used to define discrete stable distributions in the previous Sections is not the only possibility. As was showed in Chapter \ref{ch_definitions} we can consider also a $\mathcal{T}$ thinning operator (of Chebyshev type) given by the following probability generating function

\begin{equation}\label{Qz_mpds}
\Qcal(z) = \left(\frac{2\left(b+T_p\left(\frac{(1+b)z^m-2b}{2-(1+b)z^m}\right)\right)}{(1+b)\left(1+T_p\left(\frac{(1+b)z^m-2b}{2-(1+b)z^m}\right)\right)}\right)^{1/m},
\end{equation}
where $p \in (0,1)$, $b \in (-1,1)$ and $m \in \N$, and $T_p(x) = \cos\left(p \arccos x \right).$

%\Section{Modified positive discrete stable distributions}
%
%\begin{definition}\label{def_MPDS}
%Let $X, X_1, X_2, \dots, X_n$ be i.i.d.~non-negative integer-valued random variables with values in $m\N_0$, where $m \in \N$. Assume that for every $n \in \N$ there exists a~constant $p \in (0,1)$ such that
%\begin{equation}\label{eq_defmpds}
%X \stackrel{d}{=} \sum_{j=1}^{n} \widetilde{X}_j, \quad \text{where} \quad \widetilde{X}_j = \sum_{i=1}^{X_j} \varepsilon_i^{(j)},
%\end{equation}
%and $\varepsilon_i^{(j)}$ are i.i.d.~$\mathcal{T}(p,b,m)$ random variables. Then we say that $X$ is \textit{modified positive discrete stable random variable}.
%\end{definition}

\begin{theorem}
A non-negative integer-valued random variable $X$ is positive discrete stable with $\mathcal{T}$ thinning operator if and only if its probability generating function is given as
\begin{equation}\label{pgf_mpds}
\Pcal(z) = \exp\left\{-\lambda\left(\arccos \frac{(1+b)z^m -2b}{2-(1+b)z^m}\right)^{\gamma}\right\} \quad \text{with} \quad \gamma \in (0,2], \;  \lambda > 0,\; b \in (-1,1), \; m \in \N. 
\end{equation}
\end{theorem}

\begin{proof}
Let $h(z) = \log \Pcal(z)$. From Proposition \ref{prop_defPDS} it follows that $X$ is positive discrete stable if and only if $h(z) = n h(\Qcal(z))$ for all $n$, where $\Qcal$ is as in \eqref{Qz_mpds}. Set $$h(z) = -\lambda\left(\arccos \frac{(1+b)z^m -2b}{2-(1+b)z^m} \right)^{\gamma}$$ and select $\gamma$ such that $1/p^\gamma = n.$  Then
\begin{align*}
n h(\Qcal(z)) & = - \lambda n \left(\arccos \frac{(1+b)\Qcal(z)^m -2b}{2-(1+b)\Qcal(z)^m}\right)^{\gamma} \\
& = -\lambda n \left(\arccos T_p\left(\frac{(1+b)z^m-2b}{2-(1+b)z^m}\right) \right)^{\gamma} \\
& = -\lambda n \left(p  \arccos \frac{(1+b)z^m -2b}{2-(1+b)z^m}\right)^{\gamma} \\ 
& = h(z).
\end{align*}
\end{proof}
We will denote the discrete stable distribution with Chebyshev thinning operator $\mathcal{T}$ and with parameters $\gamma \in (0,2], \lambda>0$, $b \in (-1,1)$ and $m \in \N$, by $\mathcal{T}\pds(\gamma,\lambda,b,m)$. If $m$ is omitted then $m=1$. If moreover $b$ is omitted we will understand that $b=0$.

\subsection{Characterizations}

\begin{theorem}\label{char_theorem_Tpds}
Let $\gamma' \in (0,2]$ and $\gamma \in (0,1]$ and assume that $\gamma' \leq 2\gamma$. Let $\s_\gamma$ be a~$\gamma$-stable random variable with Laplace transform $\exp\{-u^{\gamma}\}$. Then 
$$\mathcal{T}\pds(\gamma',\lambda,b) \stackrel{d}{=} \mathcal{T}\pds\left(\gamma'/\gamma, \lambda^{1/\gamma} \s_{\gamma},b\right).$$
\end{theorem}

\begin{proof}
For sake of simplicity we will do the proof only for the case $b=0$. The case $b\ne 0$ can be proved in the same way. The probability generating function of $X \sim \mathcal{T}\pds\left(\gamma'/\gamma, \lambda^{1/\gamma} \s_{\gamma}\right)$ is computed as
\begin{align*}
\Pcal(z) &  = \E z^{X} = \E\exp\left\{-\lambda^{1/\gamma} \s_{\gamma} \left(\arccos \frac{z}{2-z}\right)^{\gamma'/\gamma}\right\} \\
\intertext{and using the Laplace transform formula for $\s_{\gamma}$ we have} 
\Pcal(z) & = \exp\left\{-\lambda \left(\arccos \frac{z}{2-z}\right)^{\gamma'} \right\} .
\end{align*}
This is the probability generating function of $\mathcal{T}\pds(\gamma',\lambda)$.
\end{proof}

\begin{corollary}
Let $Y, Y_1, Y_2, \dots$ be a~sequence of i.i.d. random variables with probability generating function $$\Pcal(z) = 1-\tfrac{1}{\pi} \arccos \frac{(1+b)z-2b}{2-(1+b)z}.$$ Let $N$ be a~random variable, independent of the sequence $Y_1,Y_2, \dots$, with Poisson distribution with random intensity $\lambda^{1/\gamma} \pi \s_{\gamma}$, where  $\gamma \in (0,1]$ and $\s_{\gamma}$ is a~$\gamma$-stable random variable with Laplace transform $\exp\{-u^{\gamma}\}$. A random variable $X$ is positive discrete stable $\mathcal{T}\pds(\gamma,\lambda,b)$ if and only if
$$X \stackrel{d}{=} \sum_{j=1}^N Y_j.$$
\end{corollary}
\begin{proof}
Let $X = \sum_{j=1}^N Y_j$. Then $X$ is a~compound Poisson random variable with random intensity $\lambda^{1/\gamma}\pi \s_{\gamma}$ and jumps $Y_1, Y_2, \dots$ with characteristic function $$g(t) = 1-\tfrac{1}{\pi} \arccos \frac{(1+b)e^{\im t}-2b}{2-(1+b)e^{\im t}}.$$ The characteristic function of a~compound Poisson random variable with intensity $\tau$ and characteristic function of jumps $h(t)$ is $\exp\{-\tau(1-h(t))\}.$ Therefore $X$ is in fact $\mathcal{T}\pds(1,\lambda^{1/\gamma}\s_{\gamma},b)$.  We thus obtain the result from the previous Theorem \ref{char_theorem_Tpds} with $\gamma' = \gamma$.
\end{proof}

\subsection{Continuous analogies}

Let us consider a~$\mathcal{T}$ positive discrete stable random variable $X \sim \mathcal{T}\pds(\gamma,\lambda,b)$. We are interested in the limit distribution of a~random variable $X^a = aX$, where $a \downarrow 0$. We show that the limit distribution is in fact $\alpha$-stable with index of stability $\alpha = \gamma/2$ and with skewness $\beta = 1$.

\begin{theorem}
Let $X$ be a~random variable with probability generating function $$\Pcal(z) = \exp\left\{-\lambda \left(\arccos \frac{(1+b)z-2b}{2-(1+b)z} \right)^{\gamma}\right\}, \quad \gamma \in (0,2], \;  \lambda>0, \; b \in (-1,1).$$ Let $X^a = aX$ and assume that $\lambda =  \frac{\sigma}{a^{\gamma/2}}$. Then the characteristic function of $X^a$ converges pointwise to the characteristic function of $\alpha$-stable distribution,
\begin{multline*}
f^a(t) = \exp\left\{-\lambda \left(\arccos \frac{(1+b)e^{\im at}-2b}{2-(1+b)e^{\im at}} \right)^{\gamma}\right\} \\ 
\longrightarrow \exp\left\{-\sigma 2^{\gamma} \cos \frac{\pi \gamma}{4} \left(\frac{1+b}{1-b}\right)^{\gamma/2} |t|^{\gamma/2} \left(1-\im \, \mathrm{sign}(t) \tan \frac{\pi \gamma}{4} \right)\right\}. 
\end{multline*}
\end{theorem}

\begin{proof}
For sake of simplicity we will do the proof only for $b=0$. The characteristic function of $X^a$ can be approximated as
\begin{align*}
\log f^a(t)& = -\lambda \left(\arccos \frac{e^{\im at}}{2-e^{\im at}} \right)^{\gamma} \\
& \approx -\lambda \left(\arccos \frac{1+\im at}{1-\im at} \right)^{\gamma}, \quad \text{as} \quad a~\to 0. 
\end{align*}
Moreover $\arccos(z) \approx \sqrt{2} \sqrt{1-z}$ as $z \to 1$. We have $$1- \frac{1+\im at}{1-\im at} = \frac{- 2 \im at}{1-\im at}.$$ Put together we obtain 
\begin{align*}
\log f^a(t) & \approx - \frac{\sigma}{a^{\gamma/2}} \left(2 \sqrt{\frac{- \im at}{1-\im at}}\right)^{\gamma} \quad \text{as} \quad a \to 0\\
& \to - \sigma 2^{\gamma} (-\im t)^{\gamma/2}, \quad  \text{as} \quad a \to 0. 
\end{align*}
Moreover we have $(-\im t)^{\gamma/2} =  \cos \frac{\pi \gamma}{4} |t|^{\gamma/2} \left(1-\im \, \mathrm{sign}(t) \tan \frac{\pi \gamma}{4} \right)$. The proof is therefore completed.
\end{proof}

\section*{Acknowledgements}
The paper was partially supported by Czech Science Foundation under the grants P402/12/12097 and P203/12/0665. The support by Mobility fund of Charles University in Prague and Karel Urb\'anek endowment fund is gratefully acknowledged.

%\bibliography{mybib}{}

\begin{thebibliography}{}

\bibitem[Christoph and Schreiber, 1998]{christoph}
Christoph, G. and Schreiber, K. (1998).
\newblock Discrete stable random variables.
\newblock {\em Statistics $\&$ Probability Letters}, 37(3):243--247.

\bibitem[Christoph and Wolf, 1992]{christoph_wolf}
Christoph, G. and Wolf, W. (1992).
\newblock {\em Convergence theorems with a stable limit law}, volume
  Mathematical Research, vol. 70.
\newblock Akademie Verlag.

\bibitem[Devroye, 1993]{devroye}
Devroye, L. (1993).
\newblock A triptych of discrete distributions connected to stable law.
\newblock {\em Statistics $\&$ Probability Letters}, 18:349--351.

\bibitem[Doray et~al., 2009]{doray}
Doray, L.~G., Jiang, S.~M., and Luong, A. (2009).
\newblock Some simple method of estimation for the parameters of the discrete
  stable distribution with the probability generating function.
\newblock {\em Communications in Statistics -- Simulation and Computation},
  38(9):2004--2017.

\bibitem[Erd{\'e}lyi et~al., 1953a]{erdelyiII}
Erd{\'e}lyi, A. et~al. (1953a).
\newblock {\em Higher transcendental functions}, volume~2.
\newblock McGraw-Hill, New York.

\bibitem[Erd{\'e}lyi et~al., 1953b]{erdelyiI}
Erd{\'e}lyi, A. et~al. (1953b).
\newblock {\em Higher transcendental functions}, volume~1.
\newblock McGraw-Hill, New York.

\bibitem[Faa~di Bruno, 1857]{bruno}
Faa~di Bruno, C.~F. (1857).
\newblock Note sur une nouvelle formule de calcul diff{\'e}rentiel.
\newblock {\em Quarterly J. Pure Appl. Math}, 1:359--360.

\bibitem[Feller, 1968]{feller1}
Feller, W. (1968).
\newblock {\em An introduction to probability theory and its applications},
  volume~1.
\newblock John Wiley \& Sons.

\bibitem[Feller, 1970]{feller2}
Feller, W. (1970).
\newblock {\em An introduction to probability theory and its applications},
  volume~2.
\newblock John Wiley \& Sons.

\bibitem[Gnedenko and Kolmogorov, 1949]{gnedenko_kolmogorov_russian}
Gnedenko, B.~V. and Kolmogorov, A.~N. (1949).
\newblock {\em Limit distributions for sums of independent random variables (in
  Russian)}.
\newblock Gostekhizdat, Moscow.

\bibitem[Ibragimov and Linnik, 1971]{ibragimov}
Ibragimov, I.~A. and Linnik, Y.~V. (1971).
\newblock {\em Independent and stationary sequences of random variables}.
\newblock Wolters-Noordhoff, Groningen.

\bibitem[Johnson et~al., 2005]{johnson}
Johnson, N.~L., Kemp, A.~W., and Kotz, S. (2005).
\newblock {\em Univariate discrete distributions}, volume 444.
\newblock John Wiley \& Sons.

\bibitem[Kakosyan et~al., 1984]{kakosyan}
Kakosyan, A.~V., Klebanov, L.~B., and Melamed, I.~A. (1984).
\newblock {\em Characterization of distributions by the method of intensively
  monotone operators}.
\newblock Springer.

\bibitem[Khintchine, 1938]{khintchine}
Khintchine, A.~Y. (1938).
\newblock {\em Limit laws of sums of independent random variables (in
  Russian)}.
\newblock ONTI, Moscow.

\bibitem[Klebanov, 2003]{klebanov}
Klebanov, L.~B. (2003).
\newblock {\em Heavy tailed distributions}.
\newblock Matfyzpress.

\bibitem[Klebanov et~al., 2012]{klebanov2012}
Klebanov, L.~B., Kakosyan, A.~V., Rachev, S.~T., and Temnov, G. (2012).
\newblock On a class of distributions stable under random summation.
\newblock {\em Journal of Applied Probability}, 49(2):303--318.

\bibitem[Klebanov and Sl\'amov\'a, 2013]{slamova2012}
Klebanov, L.~B. and Sl\'amov\'a, L. (2013).
\newblock Integer valued stable random variables.
\newblock {\em Statistics $\&$ Probability Letters}, 83(6):1513--1519.

\bibitem[L{\'e}vy, 1925]{levy}
L{\'e}vy, P. (1925).
\newblock {\em Calcul des probabilit{\'e}s}.
\newblock Gauthier-Villars Paris.

\bibitem[L{\'e}vy, 1937]{levy1937}
L{\'e}vy, P. (1937).
\newblock {\em Th{\'e}orie de l'addition des variables al{\'e}atoires},
  volume~1.
\newblock Gauthier-Villars Paris.

\bibitem[Marcheselli et~al., 2008]{marcheselli}
Marcheselli, M., Baccini, A., and Barabesi, L. (2008).
\newblock Parameter estimation for the discrete stable family.
\newblock {\em Communications in Statistics?heory and Methods}, 37(6):815--830.

\bibitem[Phillips, 1978]{phillips}
Phillips, M. (1978).
\newblock Sums of random variables having the modified geometric distribution
  with application to two-person games.
\newblock {\em Advances in Applied Probability}, pages 647--665.

\bibitem[Price, 1965]{price}
Price, D. J. d.~S. (1965).
\newblock Network of scientific papers.
\newblock {\em Science}, 149:510?515.

\bibitem[Rivlin, 1974]{rivlin}
Rivlin, T.~J. (1974).
\newblock {\em The Chebyshev polynomials}.
\newblock John Wiley $\&$ Sons.

\bibitem[Samorodnitsky and Taqqu, 1994]{samorodnitsky}
Samorodnitsky, G. and Taqqu, M.~S. (1994).
\newblock {\em Stable non-gaussian random processes: Stochastic models with
  infinite variance}.
\newblock Chapman $\&$ Hall.

\bibitem[Sl\'amov\'a and Klebanov, 2012]{slamovaMME}
Sl\'amov\'a, L. and Klebanov, L.~B. (2012).
\newblock Modelling financial returns with discrete stable distributions.
\newblock In {\em Proceedings of 30th International Conference Mathematical
  Methods in Economics}, pages 805--810. Silesian University in Opava, School
  of Business Administration in Karviná.

\bibitem[Sl{\'a}mov{\'a} and Klebanov, 014b]{slamova2014}
Sl{\'a}mov{\'a}, L. and Klebanov, L.~B. (\noop{3002}2014b).
\newblock On discrete approximations of stable distributions.
\newblock arXiv preprint {\tt arXiv:1403.3671 [math.PR]}.

\bibitem[Steutel and van Harn, 1979]{steutel}
Steutel, F.~W. and van Harn, K. (1979).
\newblock Discrete analogues of self-decomposability and stability.
\newblock {\em Annals of Probabability}, 7(5):893--899.

\bibitem[Uchaikin and Zolotarev, 1999]{uchaikin}
Uchaikin, V.~V. and Zolotarev, V.~M. (1999).
\newblock {\em Chance and Stability: {S}table Distributions and their
  applications}.
\newblock Walter de Gruyter.

\bibitem[Zipf, 1949]{zipf}
Zipf, G.~K. (1949).
\newblock {\em Human behavior and the principle of least effort}.
\newblock Addison--Wesley Press.

\bibitem[Zolotarev, 1986]{zolotarev}
Zolotarev, V.~M. (1986).
\newblock {\em One-dimensional stable distributions}, volume~65.
\newblock American Mathematical Soc.

\end{thebibliography}

 \newcommand{\noop}[1]{}

\end{document}